\documentclass[11pt]{amsart}
\usepackage{mymacros}
\usepackage{textcomp,geometry,graphicx,enumerate,fullpage}
\usepackage{mathrsfs}
\usepackage{booktabs}
\usepackage{setspace}
\usepackage{color}

\title{Semidefinite Relaxations of Products of Nonnegative Forms on the Sphere}

\author{Chenyang Yuan \and Pablo A. Parrilo}

\newtheorem{theorem}{Theorem}
\numberwithin{theorem}{section}
\newtheorem{definition}[theorem]{Definition}

\newtheorem{proposition}[theorem]{Proposition}

\newtheorem{fact}[theorem]{Fact}

\newtheorem{example}[theorem]{Example}
\newtheorem{conjecture}[theorem]{Conjecture}

\DeclareMathOperator{\rel}{rel}

\DeclareMathOperator{\srel}{\mathscr{S}}
\DeclareMathOperator{\GammaText}{Gamma}
\newcommand{\cnormal}{\mathcal{N}_{\C}}

\newcommand{\knormal}{\mathcal{N}_{\K}}
\newcommand{\cA}{\mathcal{A}}
\newcommand{\defn}{\coloneqq}
\newcommand{\eps}{\epsilon}

\begin{document}
\maketitle

\begin{abstract}
  We study the problem of maximizing the geometric mean of $d$ low-degree
  non-negative forms on the real or complex sphere in $n$ variables. We show
  that this highly non-convex problem is NP-hard even when the forms are
  quadratic and is equivalent to optimizing a homogeneous polynomial of degree
  $O(d)$ on the sphere. The standard Sum-of-Squares based convex relaxation for
  this polynomial optimization problem requires solving a semidefinite program
  (SDP) of size $n^{O(d)}$, with multiplicative approximation guarantees of
  $\Omega(\frac{1}{n})$. We exploit the compact representation of this
  polynomial to introduce a SDP relaxation of size polynomial in $n$ and $d$,
  and prove that it achieves a \emph{constant factor} multiplicative
  approximation when maximizing the geometric mean of non-negative quadratic
  forms. We also show that this analysis is asymptotically tight, with a
  sequence of instances where the gap between the relaxation and true optimum
  approaches this constant factor as $d \rightarrow \infty$. Next we propose a
  series of intermediate relaxations of increasing complexity that interpolate
  to the full Sum-of-Squares relaxation, as well as a rounding algorithm that
  finds an approximate solution from the solution of any intermediate
  relaxation. Finally we show that this approach can be generalized for
  relaxations of products of non-negative forms of any degree.
\end{abstract}

\section{Introduction}
Sum-of-squares optimization is a powerful method of constructing hierarchies of
relaxations for polynomial optimization problems that converge to the optimal
solution at a cost of increasing computational complexity
(\cite{LasserreGlobalOptimizationPolynomials2001},
\cite{ParriloStructuredsemidefiniteprograms2000}). However, computing these
relaxations in general requires solving large instances of semidefinite programs
(SDPs), which quickly becomes computationally intractable. In particular, to
find the Sum-of-Squares decomposition of a dense degree-$d$ polynomial in $n$
variables, the input size alone is of order $n^{O(d)}$, which is exponential in
the degree.

In this paper, we introduce a series of Sum-of-Squares based algorithms to
efficiently approximate a class of dense polynomial optimization problems where
the polynomials have \emph{high degree} (where the degree is comparable to the
number of variables) but are \emph{compactly represented} (meaning that they can
be efficiently evaluated). One example of such a polynomial is the determinant
of a $n \times n$ matrix, a degree $n$ polynomial in its $n^2$ entries (thus
having exponentially many coefficients), but can be efficiently computed in
polynomial time. The class of polynomials we study in this paper is constructed
by taking the product of low-degree non-negative polynomials. For the most of
the paper, we will focus on the product of positive semidefinite (PSD) forms,
corresponding to the product of degree-2 non-negative polynomials.
\begin{definition}
  Let $\cA = (A_1, \ldots, A_d)$ where $A_i \in \K^{n\times n}$ be
  symmetric/Hermitian PSD matrices, where $\K = \R \text{ or } \C$. Then
  \begin{align*}
    p(x) = \prod_{i=1}^d \dotp{x, A_i x},
  \end{align*}
  a degree-$2d$ polynomial of $n$ variables, is a product of PSD forms.
\end{definition}
Maximizing the product of PSD forms over the sphere generalizes many different
problems in optimization, such as Kantorovich's inequality, optimizing monomials
over the sphere, linear polarization constants for Hilbert spaces, approximating
permanents of PSD matrices, portfolio optimization, and can also be interpreted
as computing the Nash social welfare for agents with polynomial utility
functions. It also has connections to bounding the relative entropy distance
between a quadratic map and its convex hull. These applications will be further
elaborated in Section \ref{sec:motivation}. We also prove in Section
\ref{sec:hardness} that this problem is NP-hard when $d = \Omega(n)$, using a
reduction to hardness of approximation of \textsc{MaxCut}. Since $d$ can be much
greater than $n$, in order to normalize for $d$ we define our objective to be
the geometric mean of quadratic forms:
\begin{align} \label{eq:opt-prod-psd}
  \textsc{Opt}(\cA) \defn \max_{x \in \K,\, \norm{x}=1} \,
  \paren{\prod_{i=1}^d \dotp{x, A_i x}}^{1/d}.
\end{align}
Sum-of-Squares optimization allow us to create a hierarchy of algorithms of
increasing complexity that give better bounds for \eqref{eq:opt-prod-psd}. In
general, if the objective is a degree $2d$ polynomial, the lowest level of the
hierarchy is a degree-$d$ Sum-of-Squares relaxation. This relaxation for
\eqref{eq:opt-prod-psd} is written as follows:
\begin{align} \label{eq:sos-relax-full}
  \textsc{OptSOS}_d(\cA) \defn
  \min\, \gamma^{1/d} \quad \text{s.t.} \quad
  \gamma \norm{x}^{2d} - \prod_{i=1}^d \dotp{x, A_i x} \text{ is a sum of squares},
\end{align}
where a polynomial $f(x)$ is a sum of squares if there exist polynomials
$s_i(x)$ so that $f(x) = \sum_i s_i(x)^2$. The constraint that a degree $d$
polynomial in $n$ variables is a sum of squares can be represented by a SDP of
size $n^{O(d)}$. Although techniques exist for reducing the size of this
representation for sparse polynomials \cite{KojimaSparsitysumssquares2005} and
polynomials with symmetry \cite{gatermann_symmetry_2004}, the polynomial $p(x)$
may not have these properties. Thus $\textsc{OptSOS}_d(\cA)$ requires solving a
SDP of size $n^{O(d)}$. However, because of the compact representation of this
polynomial, one can perhaps hope to do better. In this paper we first present a
SDP-based relaxation of $\textsc{Opt}(\cA)$ as well as a rounding algorithm for
this relaxation.
\begin{definition}[Semidefinite relaxation of \eqref{eq:opt-prod-psd}]
  \label{def:sdp-relax}
  We define $\textsc{OptSDP}(\cA)$ to be the optimum of the following SDP-based
  relaxation of \eqref{eq:opt-prod-psd}:
  \begin{align} \label{eq:dual-sdp-general}
    \textsc{OptSDP}(\cA) \defn
    \max_X
    \paren{\prod_{i=1}^d \dotp{A_i, X}}^{1/d}
    \, \mbox{ s.t. } \,
    \left\{
    \begin{array}{rl}
      X &\succeq 0\\
      \Tr(X) &= 1
    \end{array} \right. ,
  \end{align}
  where $X$ is symmetric when $\K = \R$ and Hermitian when $\K = \C$.
\end{definition}
This relaxation comes from writing $\dotp{x, A_i x} = \dotp{A_i, xx^\dagger}$ in
\eqref{eq:opt-prod-psd} and relaxing the rank-1 matrix $xx^\dagger$ to the
semidefinite variable $X$. Finding the value of this relaxation involves solving
a SDP with $O(n^2+d)$ variables and $O(n^2d)$ constraints, compared to the
Sum-of-Squares relaxation \eqref{eq:sos-relax-full} which involves solving a SDP
of size $n^{O(d)}$. The trade-off is that this relaxation is weaker than
Sum-of-Squares (Proposition \ref{prop:sos1ged}):
\begin{align*}
  \textsc{Opt}(\cA) \le \textsc{OptSOS}_d(\cA) \le \textsc{OptSDP}(\cA).
\end{align*}
Nevertheless, we show that its approximation factor is bounded by a constant,
compared to the worst case $\frac{1}{n}$ approximation factor of general
polynomial optimization algorithms
(\cite{BhattiproluWeakDecouplingPolynomial2017},
\cite{doherty_convergence_2012}). This comes from analyzing the following
rounding algorithm which produces a feasible solution to \eqref{eq:opt-prod-psd}
given an optimum solution $X^*$ to \eqref{eq:dual-sdp-general}: Sample
$y \sim \mathcal{N}_{\K}(0, X^*)$ and return $x = y/\norm{y}$, where
$\mathcal{N}_{\K}$ is a real/complex multivariate Gaussian distribution (see
Definition \ref{def:multivariate-gaussian}).  The following theorem bounds the
multiplicative approximation factor of the relaxation $\textsc{OptSDP}(\cA)$.
\begin{theorem} \label{thm:approx-factor-prod-psd}
  Suppose there is an optimal solution $X^*$ to \eqref{eq:dual-sdp-general}
  with $\rank(X^*) = r$. Let
  \begin{align*}
    L_r(\K) &= \
    \begin{cases}
       \gamma + \log 2 + \psi\pfrac{r}{2} - \log\pfrac{r}{2}  < 1.271 & \text{ if } \K = \R \\
       \gamma + \psi(r) - \log(r)  < 0.578 &\text{ if } \K = \C
    \end{cases},
  \end{align*}
  where $\psi(x) = \frac{d}{dx} \log \Gamma(x)$ is the digamma function. Then
  \begin{align*}
    e^{-L_r(\K)} \textsc{OptSDP}(\cA) \le \textsc{Opt}(\cA) \le \textsc{OptSDP}(\cA),
  \end{align*}
  which gives us a multiplicative approximation factor of $e^{-L_r(\K)}$.
\end{theorem}
Since $\lim_{r \rightarrow \infty} \psi(r) - \log(r) = 0$, the approximation
factor is at least $0.2807$ when $\K = \R$ and $0.5614$ when $\K = \C$, and can
be improved if we can further bound $\rank(X^*)$.
In particular, since $L_1(\K) = 0$, the rounding algorithm recovers the exact
solution when $\rank(X^*)=1$. In section \ref{sec:sdp-relax} we explore a few
cases where this relaxation is exact, showing that the relaxation
\eqref{eq:dual-sdp-general} is able to exactly recover Kantorovich's inequality
(Example \ref{ex:kantorovich-proof}), as well as find the exact optimal solution for
optimizing any monomial over the sphere (Section \ref{sec:monomial-max}).

Using a connection to linear polarization constants (Section \ref{sec:int-gap}),
we show that there exists an asymptotically tight integrality gap instance where
the gap between $\textsc{Opt}(\cA)$ and $\textsc{OptSDP}(\cA)$ approaches the
approximation factor $e^{-L_r(\K)}$ as $n$ and $d$ approaches infinity. The
intuition is to choose $A_i = v_i v_i^\dagger$ to be rank-1, where $v_i$ are
symmetrically distributed on the sphere. Because of symmetry, the rounding
algorithm on this instance will sample a uniformly random point on the sphere,
completely ignoring the structure of the problem. We plot an example of such a
symmetric polynomial in Figure \ref{fig:plots}.

This also motivates the need for higher-degree relaxations that perform better
than \eqref{eq:dual-sdp-general}. In Section \ref{sec:hierarchy}, we define a
series of Sum-of-Squares based relaxations computing $\textsc{OptSOS}_k(\cA)$,
which interpolates between $\textsc{OptSOS}_1(\cA) = \textsc{OptSDP}(\cA)$ and
$\textsc{OptSOS}_d(\cA)$, the full Sum-of-Squares relaxation. We also propose a
randomized rounding algorithm which allows us to sample a feasible solution from
the relaxation. Figure \ref{fig:plots} shows the distribution sampled from this
rounding algorithm for different values of $k$ for a ``worst case'' example with
multiple global optima symmetrically distributed on the sphere. We can see that
the sampled distribution concentrates towards the true optimum values as $k$
increases. We then analyze the approximation ratio of the rounding algorithm and
provide lower bounds on the integrality gap similar to the results in Section
\ref{sec:int-gap}. Next we extend this relaxation to products of general
non-negative forms. Finally in Section \ref{sec:hardness}, we prove a hardness
of approximation result for computing $\textsc{Opt}(\cA)$ by a reduction to
$\textsc{MaxCut}$.

\begin{figure}[t]
  \label{fig:plots}
  \includegraphics[width=\textwidth]{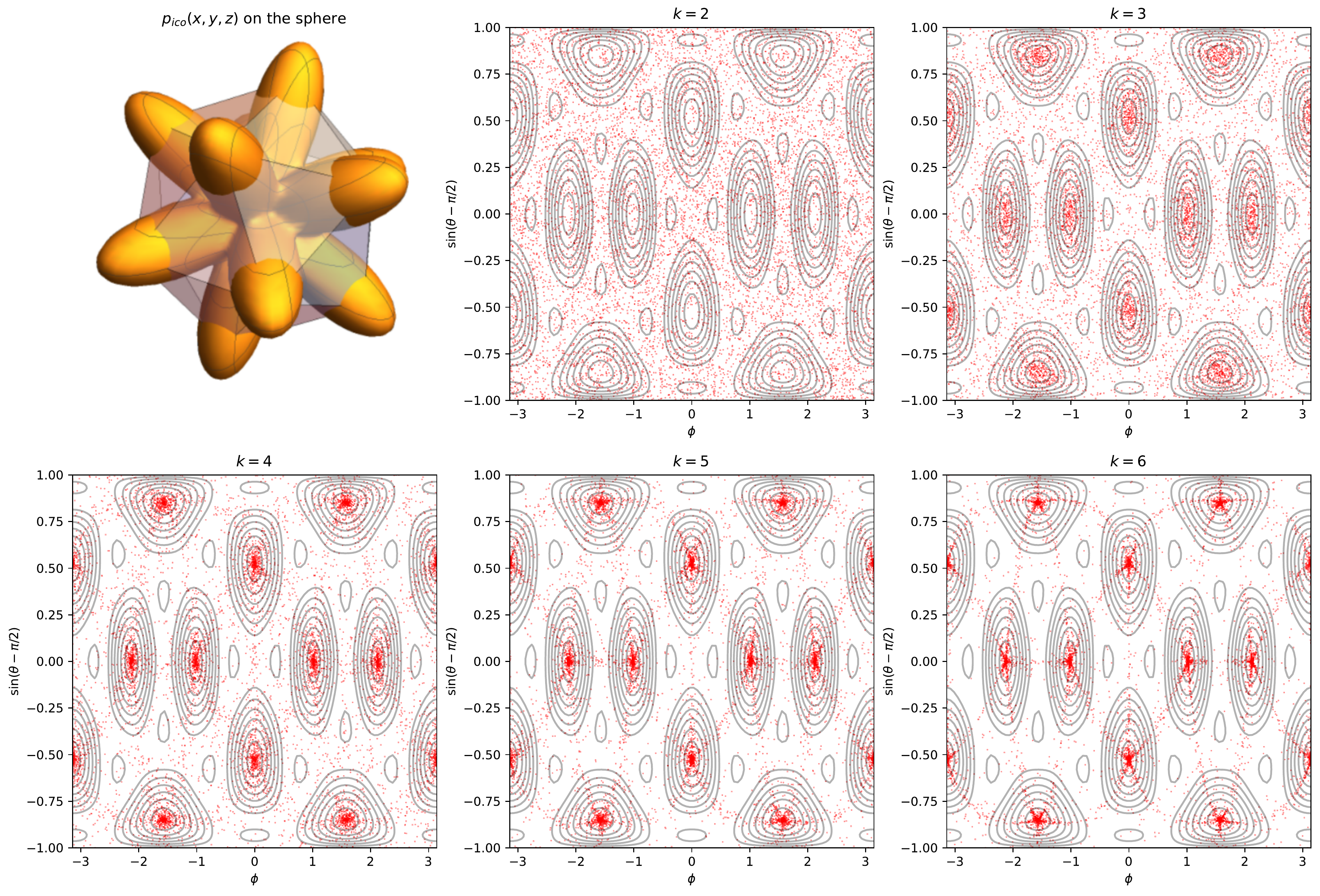}
  \centering
  \caption{$p_{\text{ico}}(x, y, z)$ is a degree 6 polynomial in 3 variables
    with icosahedral symmetry (see Example \ref{sec:example-ico} for its
    definition). The 3D plot shows the value of $p_{\text{ico}}$ on the sphere,
    superimposed on an icosahedron. We compute $\textsc{OptSOS}_k$ for
    $k=2,\ldots, 6$, relaxations of maximizing $p_{\text{ico}}(x, y, z)$ over
    the 2-sphere. The 2D plots show samples from the distribution obtained by
    the rounding algorithm to $\textsc{OptSOS}_k$, on an equal-area projection
    of the sphere. The contour plot is of $p_{\text{ico}}$ and shows its 12
    maxima on the sphere, and is overlaid on a scatter plot of 10000 points
    sampled by the rounding algorithm.}
\end{figure}

\subsection{Related Work}
There has been recent attention on problems similar to \eqref{eq:opt-prod-psd}.
The authors of this paper analyzed a special case of \eqref{eq:opt-prod-psd}
where the $A_i$ are rank-1 matrices, used in an approximation algorithm for the
permanent of PSD matrices \cite{YuanMaximizingproductslinear2021} (see Section
\ref{sec:permanent} for more details). To the best of our knowledge, the first
constant-factor approximation algorithm to \eqref{eq:opt-prod-psd} is given in
\cite{BarvinokConvexityimagequadratic2014}, and is used to prove that the
quadratic map $x \mapsto (\dotp{x, A_1 x}, \ldots, \dotp{x, A_d x})$ is close to
its convex hull in relative entropy distance. Our work improves on this
constant, and our result in Section \ref{sec:int-gap} show that it cannot be
further improved. Barvinok \cite{BarvinokFeasibilitytestingsystems1993} also
reduced the problem of certifying feasibility for systems of quadratic equations
to finding the optimum of \eqref{eq:opt-prod-psd}, and provided a polynomial
time algorithm for solving \eqref{eq:opt-prod-psd} when $d$ is fixed. A more
recent work \cite{BarvinokIntegratingproductsquadratic2020} studied a closely
related problem of approximating the integral of a product of quadratic forms on
the sphere, giving a quasi-polynomial time approximation algorithm.

For general polynomial optimization on the sphere,
\cite{doherty_convergence_2012}, \cite{BhattiproluWeakDecouplingPolynomial2017}
and \cite{Fangsumofsquareshierarchysphere2020a} gave bounds on the convergence
of the Sum-of-Squares hierarchy. These papers analyzed the convergence of higher
levels of the hierarchy (of which $\textsc{OptSOS}_d(\cA)$ is the lowest level),
proposed rounding algorithms and bounded their approximation ratios. As noted in
the introduction, these methods when applied to \eqref{eq:opt-prod-psd} takes
$n^{O(d)}$ time and only guarantees a $\Omega(1/n)$ approximation ratio, as
$p(x)$ is a high degree polynomial.

Finally, we review some strategies for speeding up Sum-of-Squares for different
polynomial optimization problems with special structure:
\begin{enumerate}
\item Solving the problem using a weakened but more computationally efficient
  version of sum of squares, for example using diagonally-dominant or
  scaled-diagonally-dominant cones instead of the positive semidefinite cone
  \cite{Ahmadiconstructionconverginghierarchies2017}. These methods typically
  sacrifice solution quality for computational tractability, but bounds on their
  approximation quality are not known.
\item Reducing the size of SDPs needed by exploiting special structure in the
  problem, such as sparsity in \cite{KojimaSparsitysumssquares2005} and
  \cite{FawziSparsesumssquares2016} or symmetry in \cite{gatermann_symmetry_2004}.
\item Using spectral methods inspired by sum of squares algorithms to solve
  average case problems \cite{hopkins_fast_2015}
  \cite{HopkinsPowerSumofSquaresDetecting2017}. They show that there exist
  spectral algorithms that are almost as good as sum of squares algorithms for
  a variety of planted problems.
\end{enumerate}
From the above works we can see that there is a trade-off between how much
structure the problem class has, how much faster the sped-up algorithm is and
how much accuracy it loses compared to running the full Sum-of-Squares
algorithm. Our work uses the compact representation of the product of
non-negative forms to arrive at the relaxation \eqref{eq:dual-sdp-general}.
This is much faster and has much better approximation guarantees than the
standard Sum-of-Squares relaxation of general polynomial optimization on the
sphere.

\subsection{Contributions}
In summary, the main contributions of this paper are:
\begin{enumerate}
\item An SDP-based relaxation \eqref{eq:dual-sdp-general} and a simple
  randomized rounding procedure that finds a feasible solution to
  \eqref{eq:opt-prod-psd}. We then prove that this is a constant-factor
  approximation algorithm to \eqref{eq:opt-prod-psd} (Theorem
  \ref{thm:approx-factor-prod-psd}).
\item Using a connection to the linear polarization constant problem (Section
  \ref{sec:lin-pol-const}) to show an integrality gap (Theorem
  \ref{thm:int-gap-SDP}) in the relaxation \eqref{eq:dual-sdp-general} that
  asymptotically matches the approximation factor shown in Theorem
  \ref{thm:approx-factor-prod-psd} as $d \rightarrow \infty$.
\item A strategy (Section \ref{sec:hierarchy}) to turn degree-2 Sum-of-Squares
  relaxations of \eqref{eq:opt-prod-psd} into degree-$k$ relaxations for any
  $k \le d$, as a way of interpolating between the relaxation
  \eqref{eq:dual-sdp-general} and the full degree-$d$ Sum-of-Squares
  relaxation. We also propose and implement a rounding algorithm to produce
  feasible solutions from these relaxations.
\item We also prove a hardness result from a reduction to $\textsc{MaxCut}$
  (Section \ref{sec:hardness}), showing that in the regime $d = \Omega(n)$, the
  problem \eqref{eq:opt-prod-psd} is NP-hard.
\end{enumerate}

\subsection{Notations}
In subsequent sections, we use $\K$ to denote either $\R$ or $\C$. For any
$x \in \C$, let $x^*$ be its complex conjugate, and $\abs{x}^2 = x x^*$. For any
matrix $A \in \K^{n \times m}$, let $A^\dagger = (A^*)^T$ be its conjugate
transpose if $\K = \C$, or its transpose if $\K = \R$. Given $a, b \in \K^n$,
let $\dotp{a, b} = a^\dagger b$ be its inner product in $\K^n$, and
$\norm{a}^2 = \dotp{a, a}$. A matrix $A$ is Hermitian if $A = A^\dagger$, and is
positive semidefinite (PSD) if in addition $x^\dagger A x \ge 0$ for all
$x \in \K$. We can also denote this as $A \succeq 0$. The $\succeq$ operator
induces a partial order called the L\"owner order, where $A \succeq B$ if
$A - B \succeq 0$.

\section{Motivation and Applications}
\label{sec:motivation}
In this section we introduce a variety of problems that can be cast into
\eqref{eq:opt-prod-psd}, maximizing the geometric mean of PSD forms over the
sphere. In particular, for a few special cases the relaxation $\textsc{OptSDP}$
is exact, corresponding to when $d=2$ (Kantorovich's inequality in Section
\ref{sec:kantorovich}) or $A_i$ are diagonal (optimizing monomials over sphere
in Section \ref{sec:monomial} and portfolio optimization in Section
\ref{sec:portfolio}). This shows that our approach generalizes many other
optimization methods and has applications to problems such as finding the linear
polarization constant of Hilbert spaces (Section \ref{sec:lin-pol-const}),
bounding the relative entropy distance between a quadratic map and its convex
hull (Section \ref{sec:quad-map-conv}), and approximating the permanent of PSD
matrices (Section \ref{sec:permanent}).

\subsection{Kantorovich's Inequality}
\label{sec:kantorovich}
\begin{proposition}[\cite{KantorovichFunctionalanalysisapplied1948}]
  \label{prop:kantorovich}
  Given a symmetric $n \times n$ positive definite matrix $A$, let
  $\lambda_1 \ge \cdots \ge \lambda_n > 0$ be its eigenvalues. Then for all
  $x \in \R^n$:
  \begin{align} \label{eq:kantorovich}
    \frac{(x^\dagger Ax)(x^\dagger A^{-1}x)}{x^\dagger x} \le \frac{1}{4}
    \paren{\sqrt\frac{\lambda_1}{\lambda_n} + \sqrt\frac{\lambda_n}{\lambda_1}}^2
  \end{align}
\end{proposition}
This inequality is used in the analysis of the convergence rate for gradient
descent (with exact line search) on quadratic objectives
$x^\dagger Ax + b^\dagger x$ (see, for example,
\cite{LuenbergerLinearnonlinearprogramming2008}). It is used to prove that the
error decreases by a factor of
$\pfrac{\lambda_1-\lambda_n}{\lambda_1+\lambda_n}^2$ with each step taken. It
can also be used to bound the efficiency of estimators in noisy linear
regression where $A$ is the covariance matrix of the noise
\cite{RaghavachariLinearProgrammingProof1986a}. The optimization problem
\eqref{eq:opt-prod-psd} is a generalization of this inequality to higher degree
products. However unlike in \eqref{eq:kantorovich} the $A_i$ may not be
simultaneously diagonalizable.

\subsection{Optimizing Monomials over the Sphere}
\label{sec:monomial}
Maximizing monomials on the sphere is a special case of \eqref{eq:opt-prod-psd}
where $A_i$ are diagonal. We can compute the exact value of the maximum of any
monomial over the sphere, and we have the following result for $\K = \R$ (a
similar result holds for $\K = \C$).
\begin{proposition} \label{prop:max-monomial-sphere}
  Let $x^\beta = \prod_{i=1}^n x_i^{\beta_i}$ be any monomial of degree
  $d = \sum_i \beta_i$. Then
  \begin{align*}
    \max_{\norm{x}=1,\, x \in \R^n} \paren{x^\beta}^{\frac{2}{d}}
    = \frac{1}{d}\prod_{i=1}^n \beta_i^{\beta_i/d}
  \end{align*}
\end{proposition}
This result is proven in Appendix \ref{sec:prop-monomial-proof}. Since we know
the exact value for this special case, it is useful to use this problem to
compare different methods of speeding up Sum-of-Squares. In particular, the
algorithms derived from Sum-of-Squares in (\cite{hopkins_fast_2015} and
\cite{BhattiproluWeakDecouplingPolynomial2017}) lose the structure of this
problem and do not return the exact optimum. We will see in section
\ref{sec:monomial-max} that our relaxation preserves this structure and is exact
in this case.

\subsection{Linear Polarization Constants for Hilbert Spaces}
\label{sec:lin-pol-const}
When all the $A_i$ in \eqref{eq:opt-prod-psd} are rank-1, the optimization
problem has connections to the linear polarization constant problem:
\begin{definition}[Linear polarization constant of a normed space]
  \label{def:lin-pol-const}
  Given a normed space $X$, let $X^*$ be its dual and
  $\mathcal{S}_X = \{ x \in X: \norm{x} = 1\}$ be the sphere with respect to the
  norm. Then the $d$-th linear polarization constant of $X$ is given by:
  \begin{align*}
    c_d(X) \defn \paren{\inf_{f_1,\ldots,f_d \in S_{X^*}} \sup_{x \in S_X}
    \abs{f_1(x)\cdots f_d(x)}}^{-1}
  \end{align*}
\end{definition}
This problem has been studied in the papers
\cite{PappasLinearpolarizationconstants2004},
\cite{Marcuslowerboundproduct1997}, and
\cite{Matolcsireallinearpolarization2006}. In particular, it is proved in
\cite{Arias-de-ReynaGaussianvariablespolynomials1998} that
$c_d(\C^d) = d^{d/2}$, but the analogous result for $\R^d$ is still a
conjecture:
\begin{conjecture}[\cite{PappasLinearpolarizationconstants2004}] \label{conj:linear-pol-const}
  Let $v_1, \ldots, v_d$ and $x$ be vectors in $\R^d$.
  \begin{align} \label{eq:max-prod-linear-forms}
    \min_{\norm{v_1}=1, \ldots, \norm{v_d}=1} \max_{\norm{x} = 1}
    \abs{\prod_{i=1}^d  \dotp{v_i, x}}
    = d^{-d/2}
  \end{align}
  And is achieved when $v_i$ are (up to rotation) the basis vectors $e_i$.
\end{conjecture}
We see that \eqref{eq:max-prod-linear-forms} is a minimax problem with its inner
maximization problem equivalent to solving the following optimization problem:
\begin{align} \label{eq:max-prod-linear-forms-squared}
  \max_{\norm{x} = 1} \paren{\prod_{i=1}^d \dotp{v_i, x}^2}^{1/d}
\end{align}
Which is exactly \eqref{eq:opt-prod-psd} with $A_i = v_i v_i^\dagger$. Exact values
for $c_d(\K^n)$ where $\K = \R \text{ or } \C$ and $d > n$ are not known, but
\cite{PappasLinearpolarizationconstants2004} computed the asymptotic value
$\lim_{d\rightarrow \infty} c_d(\K^n)^{1/d}$. We will use these results later to
construct integrality gap instances in Sections \ref{sec:int-gap} and
\ref{sec:sos-int-gap}.

\subsection{Distance of a Quadratic Map to its Convex Hull}
\label{sec:quad-map-conv}
Given $A_1, \ldots, A_d \succ 0$, let $\varphi(x): \K^n \rightarrow \R_+^d$ be a
quadratic map that maps $x$ to $(\dotp{x, A_1 x}, \ldots, \dotp{x, A_d x})$. The
convexity of the image of $\K^n$ by this map has many implications in controls
and optimization (see, for example \cite{PolikSurveySLemma2007}). The set
$\varphi(\K^n)$ is not convex in general, although it is for special cases (such
as when $d=2$). On the other hand, $\conv(\varphi(\K^n))$ has a semidefinite
representation. Barvinok \cite{BarvinokConvexityimagequadratic2014} investigated
how well $\conv(\varphi(\K^n))$ approximates $\varphi(\K^n)$ in the relative
entropy distance. Since both sets are cones in the non-negative orthant, it is
natural to compare the size of their intersection with the simplex
$\Delta_d = \{x \in \R^d \mid x_i \ge 0,\, \sum_i x_i = 1\}$.
\begin{theorem}[Theorem 1 in \cite{BarvinokConvexityimagequadratic2014}]
  Let $a \in \conv(\varphi(\R^n)) \cap \Delta_d$. Then there exists a point
  $b \in \varphi(\R^n) \cap \Delta_d$ and an absolute constant $\beta = 4.8 > 0$ such
  that
  \begin{align} \label{eq:kl-quad-map}
    \sum_{i=1}^d a_i \ln\pfrac{a_i}{b_i} \le \beta.
  \end{align}
\end{theorem}
Next we show how we can use proof of Theorem \ref{thm:approx-factor-prod-psd} to
improve the constant $\beta$, as well as extend the result to $\C^n$. Since
$a \in \conv(\varphi(\K^n))$, we can find $X \succeq 0$ such that
$a_i = \dotp{A_i, X}$. If we let $L = \sum_i A_i \succ 0$,
$A'_i = L^{-1/2}A_i L^{-1/2}$ and $X' = L^{1/2} X L^{1/2}$, we have
$\Tr(X') = \sum_i \dotp{A'_i, X'} = \sum_i a_i = 1$. Now if we sample
$z \sim \knormal(0, X')$ and let $y = z/\norm{z}$, then from the proof of Theorem
\ref{thm:approx-factor-prod-linear}, we have
\begin{align*}
  \Ex_y \bracket{\sum_i a_i \log \dotp{y, A'_i y}}
  \ge -L_r(\K) + \sum_i a_i \log \dotp{A'_i, X'},
\end{align*}
where $r$ is the rank of $X'$ satisfying $\dotp{A'_i, X'} = a_i$ and
$\Tr(X') = 1$. If we let $b_i = \dotp{y, A'_i y}$, we can choose
$\beta = L_r(\K)$ in \eqref{eq:kl-quad-map}. Furthermore, Theorem
\ref{thm:int-gap-SDP} shows that this constant is asymptotically tight.

\subsection{Permanents of PSD Matrices}
\label{sec:permanent}
Given a matrix $M \in \C^{n \times n}$, its permanent is defined to be
\begin{align*}
  \per(M) = \sum_{\sigma \in \mathbf{S}_n} \prod_{i=1}^n M_{i,\sigma(i)},
\end{align*}
Where the sum is over all permutations of $n$ elements. If $M$ is Hermitian
positive semidefinite (PSD), \cite{AnariSimplyExponentialApproximation2017a} and
\cite{YuanMaximizingproductslinear2021}
analyzed a SDP-based approximation algorithm that produces a simply exponential
approximation factor to $\per(M)$. Let $M = V^\dagger V$ and $v_i$ are the
columns of $V$. In \cite{YuanMaximizingproductslinear2021}, the problem of approximating
$\per(M)$ is related to the problem of maximizing a product of linear forms over
the complex sphere
\begin{align*}
  r(M) \defn \max_{\norm{x}^2=n} \prod_{i=1}^n \abs{\dotp{x, v_i}}^2,
\end{align*}
and its convex relaxation $\rel(M)$ (obtained in a similar manner as
\eqref{eq:dual-sdp-general}) by showing that
\begin{align*}
  \frac{n!}{n^n} r(M) \le \per(M) \le \rel(M).
\end{align*}
Thus we can approximate $\per(M)$ by analyzing the approximation quality of
$\rel(M)$ as a relaxation of $r(M)$. It is easy to see that $r(M)$ is equivalent
to a special case of \eqref{eq:opt-prod-psd} when the $A_i$ are all rank-1, and
the result of Theorem \ref{thm:approx-factor-prod-psd} applied to this problem
gives the same approximation factor to the permanent as
\cite{YuanMaximizingproductslinear2021}.

\subsection{Portfolio Optimization}
\label{sec:portfolio}
Suppose there is a collection of $n$ stocks with their returns denoted as $r$,
where $r_i > 0$ denotes the return of stock $i$ ($r_i < 1$ making a loss and
$r_i > 1$ making a profit). We wish to select a mix of these stocks to invest
in, allotting a fraction $y_i$ of our capital to stock $i$ so as to maximize our
expected return. We have the historical returns $r(1), \ldots, r(d)$ over $d$ time
periods to base our decision on. The strategy employed by
\cite{WeideStrategyWhichMaximizes1977} is to maximize the geometric mean of the
total returns
\begin{align*}
  \max_{y \ge 0, \sum_i y_i = 1} \paren{\prod_{i=1}^d \dotp{y, r(i)}}^{1/d},
\end{align*}
which can be interpreted as rebalancing the portfolio after each time
period. This is a special case of \eqref{eq:opt-prod-psd} in which $A_i$ are
diagonal matrices with $r(i)$ on the diagonal and $y_i = x_i^2$. In Section
\ref{sec:exact-diagonal} we show that in this case the relaxation
\eqref{eq:dual-sdp-general} is exact.

\subsection{Nash Social Welfare}
Suppose $x$ is an allocation of a set of divisible resources to $d$ agents each
with a non-negative utility function $A_i(x)$. We can ensure fairness by
choosing the objective function, which result in different notions of fairness,
ranging from the utilitarian $\max_x \frac{1}{d}\sum_i A_i(x)$ to egalitarian
$\max_x \min_i A_i(x)$. Interpolating between these is the Nash social welfare
objective $\max_x \paren{\prod_i A_i(x)}^{1/d}$, which is the geometric mean of
the utilities. This objective is well-studied for allocation of indivisible
items \cite{CaragiannisUnreasonableFairnessMaximum2016}, from hardness results
\cite{LeeAPXhardnessmaximizingNash2017} to constant factor approximation
algorithms \cite{AnariNashSocialWelfare2017}. In our setting, the utility function
for agent $i$ is $x^\dagger A_i x$, a non-negative quadratic form on $x$.

\section{Semidefinite Relaxation}
\label{sec:sdp-relax}
Before proving Theorem \ref{thm:approx-factor-prod-psd}, we derive our
semidefinite relaxation of the problem and give interpretations for both its
primal and dual forms. The insights gained from deriving both the primal and
dual relaxations will be helpful in Section \ref{sec:hierarchy} when
generalizing to higher-degree relaxations. Recall that the polynomial we wish to
optimize is
\begin{align*}
  p_{\cA}(x) = \prod_{i=1}^d \dotp{x, A_ix},
\end{align*}
and we want to find an upper bound of $p_{\cA}(x)^{1/d}$ on the sphere. One can
compute an upper bound using the degree-$d$ Sum-of-Squares relaxation
\eqref{eq:sos-relax-full} over the sphere,
but this involves solving a SDP of size $n^{O(d)}$, which is computationally
inefficient and does not exploit the compact representation of $p_\cA(x)$. One
computationally efficient upper bound is given by
$\prod_{i=1}^d \norm{A_i}^{1/d}$, the geometric mean of the spectral norms of
$A_i$, but it can differ from the true optimum multiplicatively by a factor of
$n^{-1/2}$ (see Proposition \ref{prop:max-monomial-sphere}). In the next few
sections, we will introduce a series of weaker but computationally more
efficient bounds, which still have good approximation guarantees.

\subsection{Quadratic Upper Bounds}
The first approach uses the arithmetic mean/geometric mean (AM/GM) inequality:
\begin{align*}
  \paren{\prod_{i=1}^d \dotp{x, A_i x}}^\frac{1}{d}
  \leq \frac{1}{d} \sum_{i=1}^d \dotp{x, A_i x} =
  \frac{1}{d} \, x^\dagger \paren{\sum_{i=1}^d A_i} x
\end{align*}
This then becomes an eigenvalue problem. Maximizing this quadratic form over the
unit sphere, we obtain the following:
\begin{proposition}
  Let $G = \sum_{i=1}^d A_i$. Then if $\norm{x} = 1$,
  \begin{align*}
    p_\cA(x) \le \paren{\frac{\lambda_{\max}(G)}{d}}^d.
  \end{align*}
\end{proposition}
This technique is powerful enough to prove Kantorovich's inequality (Proposition
\ref{prop:kantorovich}), as we will see in the following example. This is a
adaptation of Newman's proof in \cite{NewmanKantorovichinequality1960}.
\begin{example}[Proof of Kantorovich's inequality]
  \label{ex:kantorovich-proof}
  Since both $A$ and $A^{-1}$ are positive definite, we can apply the AM/GM
  inequality on $(\alpha x^\dagger Ax)(\alpha^{-1} x^\dagger A^{-1}x)$ for any $\alpha > 0$:
  \begin{align*}
    (x^\dagger Ax)(x^\dagger A^{-1}x)
    \le \frac{1}{4} \paren{x^\dagger \paren{\frac{1}{\alpha} A + \alpha A^{-1}} x}^2
    \le \frac{1}{4} \lambda_{\max}\paren{\frac{1}{\alpha} A + \alpha A^{-1}}^2.
  \end{align*}
  Without loss of generality we assume $A$ and $A^{-1}$ are diagonal, as they
  are simultaneously diagonalizable. Choosing
  $\alpha = \sqrt{\lambda_1\lambda_n}$,
  \begin{align*}
    \lambda_{\max}\paren{\frac{1}{\alpha} A + \alpha A^{-1}}
    = \max_i \paren{\frac{\lambda_i}{\sqrt{\lambda_1 \lambda_n}}
    + \frac{\sqrt{\lambda_1 \lambda_n}}{\lambda_i}}
    \le \sqrt\frac{\lambda_1}{\lambda_n} + \sqrt\frac{\lambda_n}{\lambda_1}.
  \end{align*}
  This is because $f(x) = \frac{x}{\alpha}+\frac{\alpha}{x}$ is convex on any
  nonnegative interval and a convex function on an interval is maximized at its
  endpoints.
\end{example}

\subsection{Rescaling and Semidefinite Relaxation}
\label{subsec:rescaling}
In Example \ref{ex:kantorovich-proof}, in addition to using the AM/GM
inequality, we also introduced a scaling factor $\alpha$ to strengthen the
inequality. Since the cost function is multilinear in $A_i$, we can optimize
over all possible rescalings $A_i \mapsto \alpha_i A_i$ for all $\alpha_i > 0$
where $\prod_i \alpha_i = 1$, to improve the upper bound. Furthermore, the
problem of optimizing over such scalings is also convex since a lower bound on
the concave geometric mean $\paren{\prod_{i=1}^d \alpha_i}^{1/d}$ defines a
convex set.
\begin{theorem} \label{thm:primal-sdp}
  Given $\cA = (A_1, \ldots, A_d)$, the following upper bound holds:
  \begin{align*}
    \textsc{Opt}(\cA) = \max_{\norm{x} = 1} p_\cA(x)^{1/d} \le \lambda^*,
  \end{align*}
  where $\lambda^*$ is the optimum of the following convex program:
  \begin{align} \label{eq:primal-sdp}
    \min \lambda \quad \text{s.t.}
    \quad \frac{1}{d} \sum_{i=1}^d \alpha_i A_i \preceq \lambda I_n,
    \quad \prod_{i=1}^d \alpha_i \geq 1,
    \quad \alpha_i > 0
  \end{align}
\end{theorem}
Next by taking the dual, we relate the optimum value of \eqref{eq:primal-sdp}
with that of \eqref{eq:dual-sdp-general}, which also proves the upper bound in
Theorem \ref{thm:approx-factor-prod-psd}.
\begin{theorem} \label{thm:dual-sdp}
  The following upper bound holds:
  \begin{align*}
    \textsc{Opt}(\cA) \le \textsc{OptSDP}(\cA),
  \end{align*}
  where $\textsc{OptSDP}(\cA)$ is the optimum of the following
  convex program:
  \begin{align}
    \label{eq:dual-sdp}
    \textsc{OptSDP}(\cA) \defn \max \paren{\prod_{i=1}^d \dotp{A_i, X}}^{1/d}
    \quad \mbox{ s.t. } \quad
    \left\{
    \begin{array}{rl}
      \Tr(X) &= 1 \\
      X &\succeq 0
    \end{array}
    \right.
  \end{align}
  Furthermore, \eqref{eq:dual-sdp} is dual to \eqref{eq:primal-sdp}, and
  $\textsc{OptSDP}(\cA) = \lambda^*$.
\end{theorem}
\begin{proof}
  It is clear that $\textsc{OptSDP}(\cA)$ is a rank relaxation of
  $\textsc{Opt}(\cA)$, by using the variable $X$ instead of $xx^\dagger$. To
  find the dual of \eqref{eq:primal-sdp}, we write the Lagrangian
  \begin{align*}
    \mathcal{L}(X, \gamma, \alpha, \lambda)
    = \lambda - \dotp{\lambda I - \frac{1}{d}\sum_{i=1}^d \alpha_i A_i , X}
    - \gamma \paren{\prod_i \alpha_i^{1/d} - 1}
  \end{align*}
  Solving for $\lambda$, we get the constraint $\Tr(X) = 1$. Solving for
  $\alpha_i$ and $\gamma$, we get
  \begin{align*}
    \alpha_i = \frac{\gamma}{\dotp{A_i, X}}
    \text{ and }
    \gamma = \paren{\prod_{i=1}^d \dotp{A_i, X}}^{1/d}
  \end{align*}
  And we obtain \eqref{eq:dual-sdp} after substituting these values into the
  Lagrangian.
\end{proof}
Note that the dual objective is log-concave, and it is a special case of
maximizing the determinant of a PSD matrix, which can be solved efficiently
using (for example) interior point methods
\cite{VandenbergheDeterminantMaximizationLinear1998}.

\subsection{Maximizing Monomials over the Sphere}
\label{sec:monomial-max}
To get more insight of the role the multipliers $\alpha_i$ play, we consider the
special case where $p(x) = x^{2\beta}$ is a monomial. Maximizing a monomial over
the sphere is a special case of \eqref{eq:opt-prod-psd}: for each copy of $x_i$
in $x^\beta$ (there are $d$ of these in total, corresponding to
$A_1, \ldots, A_d$), set $A_j$ to be 1 on the $i$-th diagonal entry and 0
elsewhere. Next we show that the convex relaxation in Theorem
\ref{thm:primal-sdp} achieves the true maximum value. In the relaxation there
are $d$ multipliers $\alpha_1, \ldots, \alpha_d$ associated with each copy of
$x_i$. For each $x_i$, set its multiplier to be
$\beta_i^{-1} \prod_{i=1}^n {\beta_i}^{\beta_i/d}$. Thus
\begin{align*}
  \lambda_{\max}\paren{\frac{1}{d}\sum_{j=1}^d \alpha_j A_j}
  = \lambda_{\max}\paren{\frac{1}{d}\sum_{j=1}^n \sum_{k=1}^{\beta_j} \beta_j^{-1}
  \prod_{i=1}^n {\beta_j}^{\beta_j/d} e_j e_j^\dagger}
  = \frac{1}{d} \prod_{i=1}^n {\beta_j}^{\beta_j/d}
\end{align*}
Thus the relaxation value is the same as the optimum given by Proposition
\ref{prop:max-monomial-sphere}. The multipliers $\alpha_i$ play the role of
balancing out the terms in the sum.

\subsection{Rank of Solutions}
We can bound the rank of the solution $X^*$ to the relaxation
\eqref{eq:dual-sdp} using a result by Barvinok
\cite{BarvinokCourseConvexity2002} and Pataki
\cite{PatakiRankExtremeMatrices1998}:
\begin{proposition} [Proposition 13.4 of \cite{BarvinokCourseConvexity2002}]
  \label{prop:rank-bound-sym}
  For some $r > 0$, fix $k = (r+2)(r+1)/2$ symmetric matrices
  $A_1, \ldots, A_k \in \R^{n \times n}$ where $n \ge r+2$ and $k$ real numbers
  $\alpha_1, \ldots, \alpha_k$. If there is a solution $X \succeq 0$ to the
  system:
  \begin{align*}
    \dotp{A_i, X} = \alpha_i \quad \text{ for } \quad i = 1, \ldots, k
  \end{align*}
  and the set of all such solutions is bounded, then there is a matrix
  $X_0 \succeq 0$ satisfying the same system and $\rank X_0 \le r$.
\end{proposition}
Indeed, suppose $X^*$ is an optimal solution to the relaxation
\eqref{eq:dual-sdp-general}, then any solution $X$ to the $d+1$ linear equations
$\dotp{X, A_i} = \dotp{X^*, A_i}$ and $\Tr(X) = 1$ is also optimal. Proposition
\ref{prop:rank-bound-sym}, along with an analogous result in the complex setting
\cite{AiLowRankSolutions2008}, also implies that the rank of the relaxation is
bounded by $O(\sqrt{d})$, which helps us bound the approximation factor of this
relaxation in the next section.

\subsection{Exact Relaxations}
\label{sec:exact-diagonal}

In this section we study a few special cases where the relaxation
$\textsc{OptSDP}(\cA)$ is exact. The first case is when $d = 2$, which is a direct
result of Proposition \ref{prop:rank-bound-sym} substituting in $k = 3$.
\begin{proposition}
  When $d=2$ and $\K = \R$, then $\textsc{Opt}(\cA) = \textsc{OptSDP}(\cA)$.
\end{proposition}
This also implies that the bound on Kantorovich's inequality produced by our
relaxation is tight. Next we show that the relaxation is tight when $A_i$ are
simultaneously diagonalizable. This also implies that our relaxation finds the
optimum solutions to the portfolio optimization (Section \ref{sec:portfolio}) and
optimizing monomial (Section \ref{sec:monomial}) problems.
\begin{proposition}
  Let $\cA = (A_1, \ldots, A_d)$. If all $A_i$ commute with each other, then
  $\textsc{Opt}(\cA) = \textsc{OptSDP}(\cA)$.
\end{proposition}
\begin{proof}
  Since the matrices $A_i$ commute with each other, they are simultaneously
  diagonalizable. They can be written as $A_i = U^\dagger D_i U$, where $D_i$ is
  diagonal and $U$ unitary. Then after a change of variables $x \mapsto Ux$, the
  relaxation \eqref{eq:dual-sdp-general} is equivalent to the original problem
  \eqref{eq:opt-prod-psd} with the substitution $X_{ii} = x_i^2$.
\end{proof}

\section{Rounding Algorithm and Analysis}
\label{sec:rounding-sdp}
In this section we present our randomized rounding algorithm for the relaxation
$\textsc{OptSDP}(\cA)$ \eqref{eq:dual-sdp-general}, and an analysis of its
approximation factor (Theorem \ref{thm:approx-factor-prod-psd}). First we state
some standard results about generalized Chi-squared distributions, after which
we will use these results to prove Theorem \ref{thm:approx-factor-prod-psd}.

\subsection{Background on Real and Complex Multivariate Gaussians}
In this section we will use a few results involving the expectation of functions
of real or complex multivariate Gaussian variables.
\begin{definition}[Multivariate Gaussian Random Variable]
  \label{def:multivariate-gaussian}
  Let $x \sim \knormal(0, I_n)$. If $\K = \R$, then its coordinates $x_j$ are
  i.i.d. normal random variables. If $\K=\C$, then
  $x_j = (y_j + i z_j)/\sqrt{2}$, where $y_j$ and $z_j$ are i.i.d. standard
  normal random variables.
\end{definition}
The random variable $z \sim \cnormal(0, I_n)$ is circularly symmetric, meaning
that its distribution is invariant after the transformation
$z \mapsto e^{i\theta} z$ for all $\theta \in \R$. All complex multivariate
Gaussians in this paper are circularly symmetric. Similar to real multivariate
Gaussians, a linear transform on the random vector induces a congruence
transform on the covariance matrix.
\begin{proposition}[Invariance under orthogonal/unitary transformations]
  \label{prop:knormal-linear-transform} Given $x \sim \knormal(0, \Sigma)$ and
  any matrix $A \in \K^{n \times n}$, $Ax$ has the distribution
  $\knormal(0, A \Sigma A^\dagger)$.
\end{proposition}
Thus given $A = UU^\dagger \succeq 0$, to sample $w \in \K^n$ from
$\knormal(0, A)$, we can first sample $x \sim \knormal(0, I)$, then let
$w = Ux$. The proof of this proposition and more about complex multivariate
Gaussians can be found in \cite{GallagerCircularlySymmetricGaussianrandom}. In
particular, this tells us that the distribution $\knormal(0, I)$ is invariant
under unitary transformations.

In the analysis of our rounding procedure, we use some results about the gamma
distribution.
\begin{fact}[Expectation of log of gamma random
  variable] \label{fact:ex-log-gamma} Let $X \sim \GammaText(\alpha, \beta)$ be
  drawn from the gamma distribution, with density
  $p(x; \alpha, \beta) = \Gamma(\alpha)^{-1} \beta^\alpha x^{\alpha-1}e^{-\beta
    x}$. Then
  \begin{align*}
    \Ex[\log X] = \psi(\alpha) - \log(\beta),
  \end{align*}
  where $\psi(x) = \frac{d}{dx} \log \Gamma(x)$ is the digamma function.
\end{fact}
This follows from the fact that the gamma distribution is an exponential family,
of which $\log x$ is a sufficient statistic (see section 2.2 of
\cite{KeenerTheoreticalstatisticstopics2010} for more details). Next we prove an
useful identity.
\begin{fact} \label{fact:ex-log-CN} Let
  $(z_1, \ldots, z_r) \sim \knormal(0, I_r)$,
  $\gamma = \lim_{n\rightarrow \infty} (H_n - \log n) \approx 0.577$ be the
  Euler-Mascheroni constant and
  \begin{align*}
    L_r(\K) &= \
    \begin{cases}
       \gamma + \log 2 + \psi\pfrac{r}{2} - \log\pfrac{r}{2} & \text{ if } \K = \R \\
       \gamma + \psi(r) - \log(r) &\text{ if } \K = \C
    \end{cases}.
  \end{align*}
  Then
  \begin{align*}
    \Ex \bracket{\log\paren{\frac{1}{r} \sum_{i=1}^r \abs{z_i}^2}}
    = \Ex\bracket{\log \abs{z_1}^2} + L_r(\K).
  \end{align*}
\end{fact}
\begin{proof}
  For $\K = \R$, $\sum_{i=1}^r \abs{z_i}^2$ is a chi-squared distribution with
  $r$ degrees of freedom, which is equivalent to
  $\GammaText\paren{\frac{r}{2}, \frac{1}{2}}$. Using Fact
  \ref{fact:ex-log-gamma},
  $\Ex \log\paren{\sum_{i=1}^r \abs{z_i}^2} = \psi\pfrac{r}{2} -
  \log\pfrac{r}{2}$. Since $\psi\pfrac{1}{2} = -\gamma-\log(4)$, we get
  $\Ex\log \abs{z_1}^2 = -\gamma - \log(2)$ to obtain the value of
  $L_r(\R)$. We can find $L_r(\C)$ with a similar calculation, using the fact
  that when $\K = \C$, $\sum_{i=1}^r 2\abs{z_i}^2$ is a chi-squared distribution
  with $2r$ degrees of freedom.
\end{proof}
We need the following result in our proof of Theorem
\ref{thm:approx-factor-prod-linear}.
\begin{proposition} \label{prop:expect-chi-square-lower-upper} Given
  $z \sim \knormal(0, I_r)$ and a $r\times r$ PSD matrix $M \succeq 0$ where
  $\Tr(M) = 1$,
  \begin{align*}
    \Ex_{z_1}\bracket{\log \abs{z_1}^2} \le \Ex_z\bracket{\log \dotp{z, Mz}} \le
    \Ex_{z}\bracket{\log \frac{1}{r}\sum_{i=1}^r \abs{z_1}^2}
  \end{align*}
\end{proposition}
\begin{proof}
  Because of the rotational invariance of $z$, it suffices to bound:
  \begin{align*}
    f(\lambda(M)) = \Ex_z\bracket{\log \sum_{i=1}^r \lambda_i(M) \abs{z_i}^2}.
  \end{align*}
  Since $f$ as a function of $\lambda$ is concave and symmetric on the simplex,
  it is minimized on any one of the vertices, so it is lower bounded by setting
  $\lambda = (1,0,\ldots,0)$. By a symmetry argument, $f(\lambda)$ achieves its
  maximum when $\lambda = (1/r,\ldots,1/r)$, in the center of the simplex.
\end{proof}

\subsection{Proof of Theorem \ref{thm:approx-factor-prod-psd}}
Now we will show that the value of the SDP relaxation
\eqref{eq:dual-sdp-general} is a $e^{-L_r(\K)}$ approximation of the optimum,
where $L_r(\K) \ge 0$ is upper-bounded by a fixed constant. Let $X^*$ be the
dual solution of the SDP, where $X^* \succeq 0$ and $\Tr(X^*) = 1$. Informally,
for our rounding algorithm we want to pick a vector from a distribution over the
sphere with covariance matrix $X^*$. The following theorem states the rounding
algorithm and its approximation factor.
\begin{theorem} \label{thm:approx-factor-prod-linear} Given a solution $X^*$ to
  the optimization problem \eqref{eq:dual-sdp-general} with $\rank(X^*) = r$
  that achieves value $\textsc{OptSDP}(\cA)$, we produce a feasible solution $y$
  with the following rounding procedure:
  \begin{enumerate}
  \item Sample $x\in \K^n$ uniformly at random from the multivariate Gaussian
    distribution $\mathcal{N}_\K(0, X^*)$.
  \item Return the normalized vector $y = x/\norm{x}$.
  \end{enumerate}
  If $y$ is sampled using this procedure,
  \begin{align*}
    \Ex_y \bracket{\prod_{i=1}^d \dotp{y, A_i y}^{1/d}} \ge e^{-L_r(\K)} \textsc{OptSDP}(\cA).
  \end{align*}
\end{theorem}
Since $y$ is always a feasible solution to \eqref{eq:opt-prod-psd}, we have
\begin{align*}
  \textsc{Opt}(\cA) \ge \Ex_y \bracket{\prod_{i=1}^d \dotp{y, A_i y}^{1/d}}
\end{align*}
and thus Theorem \ref{thm:approx-factor-prod-linear} implies the lower bound in
Theorem \ref{thm:approx-factor-prod-psd}.
\begin{proof}[Proof of Theorem \ref{thm:approx-factor-prod-linear}]
  Since $X^*$ is a PSD matrix it can be factored as $X^* = UU^\dagger$, where
  $U \in \K^{n \times r}$ is a rank-$r$ matrix. Another way to sample from a
  Gaussian distribution with covariance $X^*$ is to first sample
  $z \sim \mathcal{N}_\K(0, I_r)$, so that $y = Uz/\norm{Uz}$. Next we compute
  the expected value of the objective with $y$ sampled from the rounding
  procedure:
  \begin{align*}
    \Ex_z\bracket{\paren{\prod_{i=1}^d\frac{\dotp{A_iUz, Uz}}{\norm{Uz}^2}}^{1/d}}
    &= \Ex_z\bracket{\exp\paren{\frac{1}{d}\sum_{i=1}^d (\log \dotp{A_iUz, Uz} - \log \norm{Uz}^2)}} \\
    &\ge \exp\paren{\frac{1}{d}\sum_{i=1}^d (\Ex_z \log \dotp{A_iUz, Uz} - \Ex_z \log \norm{Uz}^2)},
  \end{align*}
  where we have used Jensen's inequality. Next we compute the inner expectations
  separately. Let $M = U^\dagger A_i U/\Tr(U^\dagger A_i U)$ so
  \begin{align*}
    \Ex_z \bracket{\log \dotp{A_iUz, Uz}}
    &= \Ex_z \bracket{\log \dotp{z, Mz}} + \log \Tr(U^\dagger A_i U) \\
    &\ge \Ex_z \bracket{\log \abs{z_1}^2} + \log \Tr(U^\dagger A_i U) \\
    &= \Ex_z \bracket{\log \abs{z_1}^2} + \log \dotp{A_i, X^*}.
  \end{align*}
  Since $\Tr(M) = 1$ the inequality follows from the lower bound in Proposition
  \ref{prop:expect-chi-square-lower-upper}. Next note that
  $\Tr(U^\dagger U) = \Tr(UU^\dagger) = \Tr(X^*) = 1$. Suppose
  $\lambda_1, \ldots, \lambda_n$ are the eigenvalues of $U^\dagger U$. Then
  applying the upper bound in Proposition
  \ref{prop:expect-chi-square-lower-upper} and using Fact \ref{fact:ex-log-CN}:
  \begin{align*}
    \Ex_z \bracket{\log \norm{Uz}^2} =
    \Ex_z \bracket{\log \paren{\sum_{i=1}^r \lambda_i \abs{z_i}^2}} \le
    \Ex_z \bracket{\log \paren{\frac{1}{r}\sum_{i=1}^r \abs{z_i}^2}} =
    \Ex_z \bracket{\log \abs{z_1}^2} + L_r(\K).
  \end{align*}
  Putting these together, we get that
  \begin{align*}
    \Ex_y \bracket{\prod_{i=1}^d \dotp{y, A_i y}^{1/d}}
    &\ge \exp\paren{\frac{1}{d}\sum_{i=1}^d (\log \dotp{A_i, X^*} - L_r(\K))} \\
    &= e^{-L_r(\K)}\textsc{OptSDP}(\cA).
  \end{align*}
\end{proof}

\section{Asymptotically Tight Instances} \label{sec:int-gap} An integrality gap
instance of a relaxation is a problem instance where there is a gap between the
true optimum and the SDP relaxation. In this section we provide an asymptotic
integrality gap instance for the SDP relaxation \eqref{eq:dual-sdp-general},
showing that the integrality gap approaches the approximation factor for large
$n$ and $d$. We do so by drawing a connection to the problem of linear
polarization constants on Hilbert spaces.
\begin{theorem} \label{thm:int-gap-SDP}
  For any $\epsilon > 0$, there exists $n, d$ and unit vectors
  $v_1, \ldots, v_d \in \K^n$ (where $\K = \R \text{ or } \C$) so that there is
  a gap between the true optimum of the optimization problem:
  \begin{align*}
    \textsc{Opt}(\cA) = \max_{\norm{x} = 1} \paren{\prod_{i=1}^d \abs{\dotp{x, v_i}}^2}^{1/d},
  \end{align*}
  and the SDP relaxation $\textsc{OptSDP}(\cA)$ given by
  \eqref{eq:dual-sdp-general}. This gap increases with the dimensions, so that
  for sufficiently large $n$ and $d$:
  \begin{align*}
    e^{L_n(\K)} \ge
    \frac{\textsc{OptSDP}(\cA)}{\textsc{Opt}(\cA)} \ge
    e^{L_n(\K)} - \epsilon
  \end{align*}
\end{theorem}
Intuitively, we want to choose $v_i$ respecting some symmetry, so that the
distribution on solutions returned by the SDP relaxation is as symmetrical as
possible. Thus during the rounding procedure (choosing a single solution out of
the distribution) we are forced to break this symmetry. This is where the
integrality gap instance arises. One natural choice of an instance with this
kind of symmetry is to sample each $v_i$ uniformly at random on the sphere. To
find the value of the true optimum, we use a result about the linear
polarization constants of Hilbert spaces (recall Definition
\ref{def:lin-pol-const}):
\begin{theorem}[Theorem F and 1 of \cite{PappasLinearpolarizationconstants2004}
  \footnote{The constant $L(n, \K)$ used in
    \cite{PappasLinearpolarizationconstants2004} equals to
    $-\frac{1}{2}(\log n + L_n(\K))$ in our notation. This follows from a
    straightforward application of Fact \ref{fact:ex-log-CN}.}]
  \label{thm:lin-pol-const-asymp}
  Let $\K = \R \text{ or } \C$. Then
  \begin{align*}
    \lim_{d \rightarrow \infty} c_d(\K^n)^{-2/d} = \frac{1}{n} e^{-L_n(\K)}
  \end{align*}
  and there exist a family of instances
  $\cA_d = \paren{v_1v_1^\dagger, \ldots, v_dv_d^\dagger}$ so that
  $\textsc{Opt}(\cA_d)$ converges to this value as $d \rightarrow \infty$.
\end{theorem}

\begin{proof} [Proof of Theorem \ref{thm:int-gap-SDP}]
  Applying Theorem \ref{thm:lin-pol-const-asymp}, we can find a family of
  instances $\cA_d$ so that $\textsc{Opt}(\cA_d)$ converge to
  $\frac{1}{n} e^{-L_n(\K)}$. Next, we bound $\textsc{OptSDP}(\cA_d)$. Given any
  solution $\lambda^*$ of the primal form \eqref{eq:primal-sdp},
  \begin{align*}
    \lambda^* I_n &\succeq \frac{1}{d}\sum_{i=1}^d \alpha_i v_i v_i^\dagger \\
    n \lambda^* &\ge \frac{1}{d}\sum_{i=1}^d \alpha_i
                  \ge \paren{\prod_{i=1}^d \alpha_i}^{1/d} \ge 1,
  \end{align*}
  where the inequality is obtained by taking the trace and using AM/GM. Thus
  $\textsc{OptSDP} = \lambda^* \ge 1/n$.  Putting this together with Theorem
  \ref{thm:lin-pol-const-asymp}, we have shown that there is a sequence of
  instances $\cA_d$ such that
  \begin{align*}
    \lim_{d \rightarrow \infty} \frac{\textsc{OptSDP}(\cA_d)}{\textsc{Opt}(\cA_d)}
    \ge e^{L_n(\K)}.
  \end{align*}
\end{proof}

\section{A Hierarchy of Relaxations}
\label{sec:hierarchy}
The relaxation $\textsc{OptSDP}(\cA)$ introduced in Definition
\ref{def:sdp-relax} gives us a computationally efficient algorithm to bound the
maximum of the geometric mean of PSD forms on the sphere. In this section we
discuss a few methods to strengthen $\textsc{OptSDP}(\cA)$ using Sum-of-Squares
optimization. In Section \ref{subsec:rescaling}, the SDP formulation of
$\textsc{OptSDP}(\cA)$ can be interpreted as first using the AM/GM inequality
to provide an upper bound of the original degree-$d$ polynomial in terms of a
low-degree polynomial, then optimizing over the degrees of freedom introduced by
the relaxation. We can extend this idea to higher degrees using Maclaurin's
inequality, a generalization of the AM/GM inequality. Let $\mathcal{S}_k$ be the
set of all $k$-tuples chosen from $d$ indices, of size $\binom{d}{k}$. Given
$x \in \R^d$, we define the (normalized) elementary symmetric polynomial to be
$E_k(x) = \binom{d}{k}^{-1}\sum_{I \in \mathcal{S}_k} \prod_{i \in I}x_i$ with
$E_0 = 1$. For example $E_1(x) = \frac{1}{d}\sum_{i=1}^d x_i$,
$E_2(x) = \binom{d}{2}^{-1} \sum_{i > j} x_ix_j$ and $E_d(x) = x_1\cdots
x_d$. Maclaurin's inequality states that for all $1 \le j \le k \le d$ and
$x \ge 0$,
\begin{align} \label{eq:maclaurin-sos}
  E_k(x)^{1/k} \le E_j(x)^{1/j}.
\end{align}
In particular when $j=1$ and $k=d$ we recover the AM/GM inequality. This also
generates a series of inequalities interpolating between the arithmetic and
geometric means. Since the objective of the optimization problem
\eqref{eq:opt-prod-psd} can be written as
$E_d(\dotp{x,A_1 x}, \ldots, \dotp{x,A_d x})^{1/d}$, we can get progressively
better upper bounds by optimizing
$E_k(\dotp{x,A_1 x}, \ldots, \dotp{x,A_d x})^{1/k}$ for increasing values of
$k$. Since $E_k$ is a degree-$2k$ homogeneous polynomial in $x$ we can use
Sum-of-Squares optimization to obtain bounds on its maximum.

\subsection{Background on Sum-of-Squares}
Sum-of-Squares optimization is a method of obtaining convex relaxations for
polynomial optimization problems
(\cite{LasserreGlobalOptimizationPolynomials2001},
\cite{ParriloStructuredsemidefiniteprograms2000}). Let $p(x)$ be a
degree-$2k$ polynomial. We use the notation $p(x) \succeq 0$ to denote that the
polynomial $p(x)$ can be written as a sum of squares, and $p(x) \succeq q(x)$ if
$p(x) - q(x) \succeq 0$. This can be determined by solving a SDP of size
$n^{O(k)}$. The degree-$k$ Sum-of-Squares relaxation for maximizing a
degree-$2k$ homogeneous polynomial $f(x)$ on the sphere can be written as the
following optimization problem with a Sum-of-Squares constraint:
\begin{align} \label{eq:sos-relax-general}
  \min\, \gamma \quad \text{s.t.} \quad \gamma \norm{x}^{2k} - f(x) \succeq 0
\end{align}
To take the dual of a Sum-of-Squares optimization problem, we introduce a linear
\emph{pseudoexpectation} operator $\pEx$ for each sum of squares constraint.
\begin{definition}[Homogeneous pseudoexpectation operator]
  A linear operator $\pEx : \R[x] \rightarrow \R$ on the space of degree
  degree-$2k$ homogeneous polynomials is a valid degree-$k$ homogeneous
  pseudoexpectation if $\pEx[\norm{x}^{2k}] = 1$ and $\pEx[f(x)^2] \ge 0$ for
  all degree-$k$ polynomials $f(x)$.
\end{definition}
The pseudoexpectation $\pEx$ encodes moments up to degree $2k$ and the dual of a
Sum-of-Squares problem can be viewed as optimizing over this truncated moment
sequence. Similar to a sum of squares constraint, the constraint that $\pEx$ is
a valid pseudoexpectation can be written as a SDP of size $n^{O(k)}$. Thus the
dual of \eqref{eq:sos-relax-general} can be written as:
\begin{align} \label{eq:sos-relax-dual}
  \max\, \pEx[f(x)] \quad \text{s.t.} \,
  \pEx \text{ is valid degree-$k$ homogeneous pseudoexpectation}
\end{align}
Next we provide a series of relaxations that interpolates between the
relaxations \eqref{eq:primal-sdp} and \eqref{eq:sos-relax-full}.

\subsection{Higher Degree Relaxations}
\label{sec:higher-degree-derivation}
Given an instance $\cA = (A_1, \ldots, A_d)$ of the problem
\eqref{eq:opt-prod-psd}, we can write the following relaxation of
$\textsc{Opt}(\cA)$ using Maclaurin's inequality:
\begin{align} \label{eq:sos-relax-no-alpha}
  \textsc{Opt}(\cA) \le \srel_k(\cA) \defn
  \left\{
  \begin{array}{rl}
    \min \,& \lambda^{1/k} \\
    \mbox{ s.t. }
           & \lambda\norm{x}^{2k} - E_k(\dotp{x,A_1 x}, \ldots, \dotp{x,A_d x}) \succeq 0
  \end{array}
             \right.
\end{align}
This is because
$\max_{\norm{x} = 1} E_d^{1/d} \le \max_{\norm{x} = 1} E_k^{1/k}$ follows from
\eqref{eq:maclaurin-sos} and $\srel_k(\cA)$ is a Sum-of-Squares relaxation of
the latter problem. Also as $k$ increases, the approximation improves until when
$k=d$ we get the standard degree-$d$ Sum-of-Squares relaxation
\eqref{eq:sos-relax-full}. Thus by varying $k$ we have a series of relaxations
of increasing degree.

Similar to the relaxation $\textsc{OptSDP}(\cA)$ presented in Theorem
\ref{thm:primal-sdp}, we can also use multipliers to improve $\srel_k$, arriving
at our definition for $\textsc{OptSOS}_k(\cA)$:
\begin{definition} \label{def:optsos}
  Let $\mathcal{S}_k$ be the set of all combinations of $k$-tuples from $d$
  indices, of size $\binom{d}{k}$. Then
  \begin{align} \label{eq:sos-primal}
    \setstretch{1.5}
    \begin{array}{rl}
      \textsc{OptSOS}_k(\cA) \defn
      \min \,&\lambda^{1/k} \\
      \mbox{s.t.}
             &
               \lambda\norm{x}^{2k} - \binom{d}{k}^{-1}\sum_{I \in \mathcal{S}_k} \alpha_I \prod_{i\in I} \dotp{x,A_i x}
                   \succeq 0 \\
             & \prod_{I \in \mathcal{S}_k} \alpha_I \ge 1 ,\, \alpha_I > 0
    \end{array}.
  \end{align}
\end{definition}
With this definition, $\textsc{OptSOS}_1(\cA) = \textsc{OptSDP}(\cA)$, and
$\textsc{OptSOS}_d(\cA)$ is equivalent to the degree-$d$ Sum-of-Squares
relaxation to the optimization problem
$\max_{\norm{x} = 1} \prod_{i=1}^d \dotp{x, A_i x}$. Similar to taking the dual
of $\textsc{OptSOS}_d(\cA)$ in Theorem \ref{thm:dual-sdp}, the dual of
\eqref{eq:sos-primal} is equivalent to:
\begin{align} \label{eq:sos-dual}
  \setstretch{2}
  \begin{array}{rl}
    \textsc{OptSOS}_k =
    \max \, & \paren{\prod_{I\in \mathcal{S}_k}
              \pEx_x\bracket{\prod_{i \in I} \dotp{x, A_i x}}}^{\frac{1}{d \binom{d-1}{k-1}}}\\
    \mbox{s.t.} & \pEx_x \text{ is a degree-$k$ homogeneous pseudoexpectation}
  \end{array}
\end{align}
From the above discussion, we produced a series of relaxations of increasingly
higher Sum-of-Squares degree.
\begin{proposition}
  For all $1 \le k \le d$,
  \begin{align*}
    \textsc{Opt}(\cA) \le \textsc{OptSOS}_k(\cA) \le \srel_k(\cA)
  \end{align*}
\end{proposition}
Because we are taking powers of the polynomials, it isn't immediately clear that
the values of this series of relaxations increase monotonically as we increase
the degree. Next we will prove a monotonicity result on the value of the
intermediate relaxations $\srel_k(\cA)$.
\begin{theorem}\label{thm:sos-relax-partial-order}
  Given an instance $\cA = (A_1, \ldots, A_d)$, for any
  $1 \le k < nk \le d$:
  \begin{align*}
    \srel_{nk}(\cA) \le \srel_k(\cA)
  \end{align*}
  Where $\srel_k$ is defined in \eqref{eq:sos-relax-no-alpha}.
\end{theorem}
From this we can show a partial order for relaxation values $\srel_k$, based on
the divisibility of their degrees. For example, if $d = 2^m$, Theorem
\ref{thm:sos-relax-partial-order} implies that
$\srel_1(\cA) \ge \srel_2(\cA) \ge \srel_4(\cA) \ge \cdots \ge
\srel_{2^m}(\cA)$. We use the following lemma about a Sum-of-Squares proof of
Maclaurin's inequality to prove Theorem \ref{thm:sos-relax-partial-order}.
\begin{proposition}[Lemma 3 of
  \cite{FrenkelMinkowskiinequalitysums2012}] \label{prop:partial-maclaurin-sos}
  Given $x \in \R^n$, let $s_1(x), \ldots, s_d(x) \succeq 0$ be sum of squares
  polynomials. Next let $E_k(x) = E_k(s_1(x),\ldots,s_d(x))$ be the $k$-th
  elementary symmetric polynomial in the variables $s_1(x), \ldots, s_d(x)$. For
  all $1 \le i \le j \le d-1$ the following sum of squares (in the variable $x$)
  inequality holds:
  \begin{align} \label{eq:partial-maclaurin-sos}
    E_i(x) E_j(x) \succeq E_{i-1}(x) E_{j+1}(x)
  \end{align}
  We can use \eqref{eq:partial-maclaurin-sos} to prove Maclaurin's inequality:
  \begin{align} \label{eq:scaled-maclaurin-sos}
    E_i(x)^j \succeq E_j(x)^i
  \end{align}
  As well as the following inequality:
  \begin{align} \label{eq:sos-maclaurin}
    E_m(x)^n \succeq E_{mn}(x)
  \end{align}
\end{proposition}

\begin{proof}[Proof of Theorem \ref{thm:sos-relax-partial-order}]
  Since $\srel_k$ is an optimal solution to \eqref{eq:sos-relax-no-alpha}, let
  $\lambda^*_k = \srel_k(\cA)^k$ and we have
  \begin{align*}
    \lambda_k^* \norm{x}^{2k} - E_k(\cA) \succeq 0.
  \end{align*}
  Since $\norm{x}^{2k}$ and $E_k(\cA)$ are both Sum-of-Squares polynomials in
  $x$, this implies that
  \begin{align*}
    {\lambda_k^*}^n \norm{x}^{2nk} - E_k(\cA)^n \succeq 0.
  \end{align*}
  From \eqref{eq:sos-maclaurin} we can show that
  \begin{align*}
    {\lambda_k^*}^n \norm{x}^{2nk} - E_{kn}(\cA)
    &\succeq 0.
  \end{align*}
  Since the above equation is a feasible solution to optimization problem
  \eqref{eq:sos-relax-no-alpha} with optimum $\srel_{kn}(\cA)$, we have
  $\srel_{kn}(\cA) = \paren{\lambda^*_{kn}}^{1/kn} \le \paren{\lambda_k^*}^{1/k} = \srel_k(\cA)$.
\end{proof}
If we let $k = 1$ and $n = d$, we can introduce multipliers $\alpha_i$ to this
proof to get:
\begin{proposition} \label{prop:sos1ged}
  \begin{align*}
    \textsc{OptSOS}_d(\cA) \le \textsc{OptSOS}_1(\cA) = \textsc{OptSDP}(\cA)
  \end{align*}
\end{proposition}
It is natural to ask how good an approximation $\textsc{OptSOS}_k(\cA)$ is as a
function of $k$, and how we can recover a feasible solution from the solution of
\eqref{eq:sos-primal}. We will first propose a rounding algorithm for all levels
of this relaxation that generalizes the rounding algorithm presented in Section
\ref{sec:rounding-sdp}, then analyze its approximation ratio for the case where
the relaxation is exact. Finally we show a lower bound on the integrality gap of
$\textsc{OptSOS}_k(\cA)$, and show that this bound decreases as $k$ increases.

\subsection{Rounding Algorithm}
\label{sec:rounding-alg}

Even though the higher-degree relaxations $\textsc{OptSOS}_k(\cA)$ provide upper
bounds to the true optimum $\textsc{Opt}(\cA)$, it is not immediately clear how
to produce a feasible solution to the optimization problem
\eqref{eq:opt-prod-psd}. Here we describe a general rounding procedure for
obtaining a feasible solution from each of the higher-degree relaxations. As a
generalization to the rounding algorithm for the quadratic case in Section
\ref{sec:rounding-sdp}, we first construct a PSD moment matrix $M = UU^\dagger$
(unlike the quadratic case, $M$ is chosen randomly) and generate a feasible
point $y$ on the sphere as follows:
\begin{enumerate}
\item Sample $v$ uniformly at random on $S_\K$
\item Sample $x \sim N_\K(0, M(v))$, where $M(v) = \pEx\bracket{\dotp{v, x}^{2k-2} x x^\dagger}$
\item Return $y = x/\norm{x}$
\end{enumerate}
In the proof of Theorem \ref{thm:approx-factor-prod-psd} in Section
\ref{sec:rounding-sdp}, we showed that when $k=1$, the above rounding algorithm
produces a solution that achieves a value of at least
$e^{-L_r(\K)} \textsc{OptSOS}_1(\cA)$ in expectation. This implies that
$\textsc{OptSOS}_1(\cA)$ achieves an approximation factor of at least
$e^{-L_r(\K)}$. One natural question to ask is if the higher-degree
$\textsc{OptSOS}_k(\cA)$ improves on this approximation factor.

We answer this question partially by providing a lower bound on the performance
of the rounding algorithm for instances $\cA$ where the relaxation
$\textsc{OptSOS}_k(\cA)$ is exact. Then we state a conjecture involving an
identity of pseudoexpectations which if true, the same bound applies to all
instances. Even when the relaxation is exact, this is a non-trivial
result. Since there can be exponentially many solutions to
\eqref{eq:opt-prod-psd} (see for instance the example in Section
\ref{sec:example-ico}), recovering one solution from the pseudoexpectation in
$\textsc{OptSOS}_k(\cA)$ is a tensor decomposition problem. For clarity of
exposition, we present our result for the case of $\K = \C$. We note that an
analogous result can also be proved for $\K = \R$.

\begin{theorem} \label{thm:rounding-actual} Suppose $\K = \C$ and
  $\textsc{OptSOS}_k(\cA) = \textsc{Opt}(\cA)$. Let
  \begin{align*}
    C(n, k) \defn \gamma + \frac{(1-\eps)^{n-1}(-\gamma + \log (1-\eps))}{(1-\eps - \eps/(n-1))^{n-1}} -
              \sum_{\ell=1}^{n-1} \frac{\eps (1-\eps)^{\ell-1}(\log(\eps/(n-1)) + \psi(n-\ell))}{(n-1)(1-\eps - \eps/(n-1))^\ell}
  \end{align*}
  There exists a vector $v$ so that given $y$ generated from the above rounding
  procedure,
  \begin{align} \label{eq:rounding-expectation}
    \Ex_y \bracket{\prod_{i=1}^d \dotp{y, A_i y}^{1/d}}
    \ge e^{- C(n, k)} \textsc{OptSOS}_k(\cA),
  \end{align}
  where $C_\K(n, k) \ge 0$ is bounded from above by $L_n(\C)$ and decreases with
  increasing $k$.
\end{theorem}

\begin{proof}
  First we write the expectation in exponential form and use Jensen's
  inequality:
  \begin{align*}
    \Ex_y \bracket{\prod_{i=1}^d \dotp{y, A_i y}^{1/d}}
    &= \Ex_y \bracket{\exp\paren{ \sum_{i=1}^d \dotp{y, A_i y}^{1/d}}} \\
    &\ge \exp\paren{\frac{1}{d} \sum_{i=1}^d \Ex_y \log \dotp{y, A_i y}} \\
    &= \exp\paren{\frac{1}{d} \binom{d-1}{k-1}^{-1} \sum_{I\in \mathcal{S}_k } \sum_{i \in I} \Ex_y \log \dotp{y, A_i y}}
  \end{align*}
  Next we analyse each term in the sum in the exponential. Let
  $M(v) = UU^\dagger$, and the rounding procedure is equivalent to setting
  $y = Uw/\norm{Uw}$, where $w$ is drawn from a standard multivariate complex
  Gaussian distribution.
  \begin{align*}
    \sum_{i \in I} \Ex_y \log \dotp{y, A_i y}
    &= \sum_{i \in I} \Ex_{w} \log \dotp{Uw, A_i Uw} \\
    &= \sum_{i \in I} \log{\Tr(U^\dagger A_i U)} +
      \Ex_{w}\bracket{\log w^\dagger\frac{U^\dagger A_i U}{\Tr(U^\dagger A_i U)}w}
      - \Ex_{w}\bracket{\log \norm{Uw}^2} \\
    &\ge \sum_{i \in I} \log{\dotp{A_i, M(v)}} - \gamma - \Ex_{w}\bracket{\log \norm{Uw}^2},
  \end{align*}
  where the inequality is implied by Proposition
  \ref{prop:expect-chi-square-lower-upper}. Next we bound $\log{\dotp{A_i, M}}$ in
  terms of $\Ex_{x\sim \mu}$, expectations over the distribution of solutions to
  \eqref{eq:opt-prod-psd}. We first define $M(v)$ as expectation over a reweighed
  distribution $\mu'$. Let
  \begin{align*}
    f(x) &= \frac{\dotp{v, x}^{2k-2}} {\Ex_{x \sim \mu}\bracket{\dotp{v, x}^{2k-2}}} \\
    \Ex_{x \sim \mu'}[g(x)] &= \Ex_{x \sim \mu}[f(x) g(x)],
  \end{align*}
  since $f(x) \ge 0$ and $\Ex_{x\sim \mu}[f(x)] = 1$. Then
  $M(v) = \Ex_{x \sim \mu'} \bracket{xx^\dagger}$ and we use Jensen's inequality
  to show that
  \begin{align*}
    \sum_{I \in \mathcal{S}_k}\sum_{i \in I} \log{\dotp{A_i, M(v)}}
    = \sum_{I \in \mathcal{S}_k} \sum_{i \in I}\log{\Ex_{x \sim \mu'}\bracket{\dotp{x, A_i x}}}
    \ge \Ex_{x \sim \mu'}\bracket{\sum_{I \in \mathcal{S}_k}\sum_{i \in I} \log{\dotp{x, A_i x}}}.
  \end{align*}
  Now since we are assuming that $\mu$ (and so is $\mu'$) is a distribution over
  actual solutions,
  \begin{align*}
    \sum_{I \in \mathcal{S}_k}\sum_{i \in I} \log{\dotp{x, A_i x}}
    = \log \paren{\prod_{I\in \mathcal{S}_k}
    \Ex_{x \sim \mu}\bracket{\prod_{i \in I} \dotp{x, A_i x}}},
  \end{align*}
  and since this is constant for all $x$ in the support of $\mu$ and $\mu'$, we have
  \begin{align*}
    \sum_{I \in \mathcal{S}_k} \sum_{i \in I} \log{\dotp{A_i, M(v)}}
    \ge \log \paren{\prod_{I\in \mathcal{S}_k}
    \Ex_{x \sim \mu}\bracket{\prod_{i \in I} \dotp{x, A_i x}}}.
  \end{align*}
  This completes the first part of our proof. Next we need to upper bound
  $\Ex_{w}\bracket{\log \norm{Uw}^2}$, which by results in Appendix
  \ref{sec:expected-log-chi-squared} depends on the eigenvalues of $U^\dagger U$
  which are the same as the eigenvalues of $M(v)$. Informally, with high
  probability one eigenvalue of $M(v)$ will be large while the other ones will
  be small, since by taking high powers the gap between the top eigenvalue and
  the other eigenvalues will be amplified. Thus we can use the results in
  Section \ref{sec:expected-log-chi-squared} to bound the last term. First we
  compute a lower bound for $\lambda_{\max}(M(v))$ for the case where $\K = \C$.
  \begin{align*}
    \lambda_{\max}(M(v)) =
    \lambda_{\max}\pfrac{\Ex_x\bracket{\dotp{v, x}^{2k-2}xx^\dagger}}{\Ex_x\bracket{\dotp{v, x}^{2k-2}}}
    \ge \frac{\Ex_x\bracket{\dotp{v, x}^{2k}}}{\Ex_x\bracket{\dotp{v, x}^{2k-2}}}
  \end{align*}
  Using the fact that for random variables $X$ and $Y$ where $\Pr(Y>0)=1$,
  $\Pr(X/Y \ge \Ex[X]/\Ex[Y]) > 0$, we know that there exist a $v$ such that:
  \begin{align*}
    \lambda_{\max}(M(v)) \ge
    \frac{\Ex_v\Ex_x\bracket{\dotp{v, x}^{2k}}}{\Ex_v \Ex_x\bracket{\dotp{v, x}^{2k-2}}} =
    \frac{\Ex_v[|v_1|^{2k}]}{\Ex_v[|v_1|^{2k-2}]} = \binom{n+k-1}{n-1}^{-1} \binom{n+k-2}{n-1}
    = \frac{k}{k+n-1}
  \end{align*}
  Note that this also holds if $\Ex_x$ is a pseudoexpectation instead, since we
  can interchange expectations and pseudoexpectations and
  $\Ex_v\bracket{\dotp{v, x}^{2k}} = \Ex_v[\abs{v_1}^{2k}] \norm{x}^{2k}$. If we
  let $\frac{k}{k+n-1} = 1-\eps$ and suppose that $1-\eps \ge 1/n$ (always holds
  when $k\ge 1$), then by a symmetry argument and using the concavity of
  $f(\lambda) = \Ex\bracket{\log\paren{\lambda_1 \abs{w_1}^2 + \cdots + \lambda_n
      \abs{w_n}^2}}$,
  \begin{align*}
    \Ex_{w}\bracket{\log \norm{Uw}^2}
    &= \Ex\bracket{\log\paren{\lambda_1 \abs{w_1}^2 + \cdots + \lambda_n \abs{w_n}^2}} \\
    &\le \Ex\bracket{\log\paren{\lambda_1 \abs{w_1}^2 + \frac{1-\lambda_1}{n-1} \abs{w_2}^2 + \cdots + \frac{1-\lambda_1}{n-1} \abs{w_n}^2}} \\
    &\le \Ex\bracket{\log\paren{(1-\eps) \abs{w_1}^2 + \frac{\eps}{n-1} \abs{w_2}^2 + \cdots + \frac{\eps}{n-1} \abs{w_n}^2}}.
  \end{align*}
  The last inequality arises because the expectation as a function of $\eps$ is
  monotonically increasing on the interval $[0, 1-1/n]$.  Using the result in
  Appendix \ref{sec:expected-log-chi-squared}, we get that if $k \ge 1$, there
  exists a $v$ so that
  \begin{align*}
    \Ex_z\bracket{\log \sum_i \abs{z_i}^2\lambda_i(M(v))}
    &\le \frac{(1-\eps)^{n-1}(-\gamma + \log (1-\eps))}{(1-\eps - \eps/(n-1))^{n-1}} -
      \sum_{\ell=1}^{n-1} \frac{\eps (1-\eps)^{\ell-1}(\log(\eps/(n-1)) + \psi(n-\ell))}{(n-1)(1-\eps - \eps/(n-1))^\ell} \\
    &= -\gamma + C(n, k)
  \end{align*}
\end{proof}

\begin{figure}[h]
  \centering
  \includegraphics[scale=0.4]{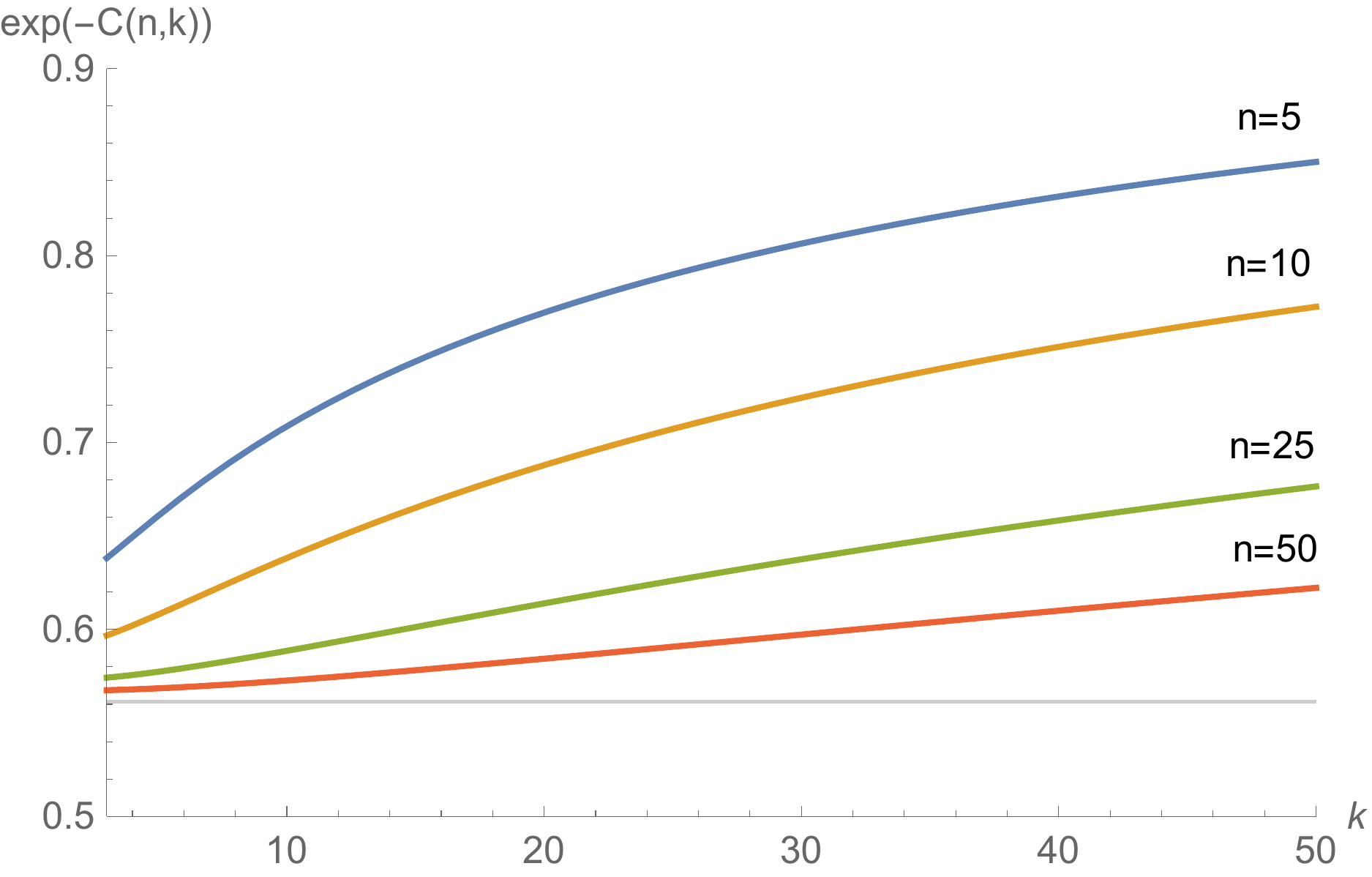}
  \caption{\label{fig:rounding} Plot of $e^{-C(n,k)}$ for different values of
    $n$ and $k$. The horizontal line shows the lower bound
    $e^{-L_r(\C)} > 0.5614$.}
\end{figure}

The analysis of Theorem \ref{thm:rounding-actual} is done assuming that the
solution to the relaxation are real distributions over solutions. To analyze the
approximation ratio we need to translate the results to pseudodistributions. We
first define
\begin{align*}
  \mu(x) \defn \frac{\dotp{v,x}^{2(k-1)}}{\pEx_x \bracket{\dotp{v,x}^{2(k-1)}}},
\end{align*}
so that $\pEx[\mu(x)] = 1$. If the following conjecture is true, then
$\textsc{OptSOS}_k(\cA)$ has an approximation factor of $e^{- C(n, k)}$.
\begin{conjecture}
  \begin{align*}
  \paren{\prod_{i=1}^d \pEx_x \bracket{\mu(x) \dotp{x, A_i x}}}^{\binom{d-1}{k-1}}
  \ge \prod_{I\in \mathcal{S}_k} \pEx_x\bracket{\prod_{i \in I} \dotp{x, A_i x}}.
\end{align*}
\end{conjecture}
For example, in the case where $k=d$, the above inequality reduces to
\begin{align*}
  \prod_{i=1}^d \pEx\bracket{\mu(x) \dotp{x, A_i x}}
  = \prod_{i=1}^d \frac{\pEx\bracket{\dotp{v,x}^{2(d-1)} \dotp{x, A_i x}}}
                       {\pEx\bracket{\dotp{v,x}^{2(d-1)}}}
  \ge \pEx\bracket{\prod_{i=1}^d \dotp{x, A_i x}}.
\end{align*}

\subsection{Example: Icosahedral Form}
\label{sec:example-ico}
Let $\phi = (1+\sqrt{5})/2$ and consider the following
degree-6 polynomial in 3 variables encoding the symmetries of the icosahedron:
\begin{align*}
  p_{\text{ico}}(x,y,z)
  = \bracket{5(2 \phi - 3) (x+\phi y)(x-\phi y)(y+\phi z)(y-\phi z)(z+\phi x)(z-\phi x)}^2.
\end{align*}
On the sphere $x^2+y^2+z^2 = 1$, $p_{\text{ico}}$ has 62 critical points: 12
maxima on the faces, 20 minima on the vertices and 30 saddle points on the edges
of the icosahedron. The normalizing constant is chosen so that
$p_{\text{ico}}(x,y,z)$ has a maximum of 1 on the sphere.  Because of its
icosahedral symmetry, $p_{\text{ico}}$ an example of a polynomial where the gap
between the SDP-based relaxation and the true optimum is large. When we solve
the relaxation $\textsc{OptSDP} = \textsc{OptSOS}_1$ for maximizing
$p_{\text{ico}}$ on the sphere, $X^*=I_3$ because of symmetry. Thus the rounding
algorithm in Section \ref{sec:rounding-sdp} reduces to sampling a uniformly
random point on the sphere, completely ignoring the structure of
$p_{\text{ico}}$. However, we can do better by solving the relaxations
$\textsc{OptSOS}_k$ for $k = 2, \ldots, 6$. The following table shows the upper
bounds obtained for different values of the relaxation parameter $k$. We can
also apply the rounding algorithm described in the previous section to this
problem, obtaining lower bounds by taking the mean of the function value from
samples returned from the rounding algorithm. From the table below we can see
the quality of the bounds increases with $k$, and when $k=6$ the relaxation is
exact.

\begin{center}
\begin{tabular}{ccc}
\toprule
$k$ & Rounding lower bound & SoS upper bound \\
\midrule
1&0.66019&1.27454 \\
2&0.65575&1.16814 \\
3&0.80480&1.10292 \\
4&0.86907&1.05821 \\
5&0.90546&1.02534 \\
6&0.92616&1.00000 \\
\bottomrule
\end{tabular}
\end{center}

Figure \ref{fig:plots} contains a 3D plot of $p_{\text{ico}}$ showing its
icosahedral symmetry, as well as 2D scatter plots of points sampled from the
rounding algorithm for $k=2, \ldots, 6$. This shows that the distribution
induced by the rounding procedure getting increasingly concentrated towards the
optimal points as the degree $k$ increases.

\subsection{Quality of Sum-of-Squares Relaxations} \label{sec:sos-int-gap}
Similar to Section \ref{sec:int-gap}, we can show a more general result, where
even with the Sum-of-Squares relaxation \eqref{eq:sos-relax-no-alpha}, there is
an integrality gap depending on the degree of relaxation.

\begin{theorem} \label{thm:int-gap-SoS} For any $k \ge 1$ and $\epsilon > 0$,
  there exists $n, d$ and unit vectors
  $v_1, \ldots, v_d \in \K^n$ (where $\K = \R \text{ or } \C$) so that there is
  a gap between the true optimum of the optimization problem:
  \begin{align*}
    \textsc{Opt}(\cA) = \max_{\norm{x} = 1} \paren{\prod_{i=1}^d \abs{\dotp{x, v_i}}^2}^{1/d},
  \end{align*}
  and the value of the degree $k$ Sum-of-Squares relaxation
  $\textsc{OptSoS}_k(\cA)$ given by \eqref{eq:sos-relax-full}:
  \begin{align*}
    \pfrac{\textsc{OptSoS}_k(\cA)}{\textsc{Opt}(\cA)} \ge
    \frac{e^{L_r(\K)}}{1 + (k-1)/n} - \epsilon
  \end{align*}
\end{theorem}
To prove this result, we need the following bound on the Sum-of-Squares
relaxation:
\begin{proposition} \label{prop:sos-relax-lower-bound}
  Given any instance $\cA = (v_1 v_1^\dagger, \ldots, v_d v_d^\dagger)$, where
  $v_1, \ldots, v_d \in \K^n$ are unit vectors, then
  \begin{align*}
    \textsc{OptSoS}_k(\cA) \ge \frac{1}{n + k - 1}.
  \end{align*}
\end{proposition}
\begin{proof}
  The Sum-of-Squares algorithm produces a certificate in the form of the
  pseudo-expectation linear operator that satisfies:
  \begin{align*}
    \pEx \bracket{\lambda^d \norm{x}^{2d} -
    E_k(\abs{\dotp{x, v_1}}^2, \ldots, \abs{\dotp{x, v_d}}^2)} \ge 0
  \end{align*}
  We can obtain a lower bound on the optimal $\lambda^*$ by taking an
  expectation over a uniform distribution on the sphere instead. For the complex
  case, We can convert each term in the expectation to an integral over a
  complex Gaussian measure $d\mu_n(x)$:
  \begin{align*}
    \int_{S_\C^{n-1}} \prod_{i=1}^k \abs{\dotp{x, v_i}}^2 dx
    = \frac{(n-1)!}{(k+n-1)!} \int_{\C^n} \prod_{i=1}^k \abs{\dotp{x, v_i}}^2 d\mu_n(x)
  \end{align*}
  Then using the integral representation of the permanent, we can rewrite the integral as
  \begin{align*}
    \int_{\C^n} \prod_{i=1}^k \abs{\dotp{x, v_i}}^2 d\mu_n(x) = \per(V^\dagger V),
  \end{align*}
  where $v_i$ are the columns of $V$. Since $V^\dagger V$ is positive semidefinite and
  has 1 on its diagonal, by Lieb's theorem
  \cite{LIEBProofsConjecturesPermanents1966} its permanent is at least 1. Therefore
  \begin{align*}
    {\lambda^*}^d
    \ge \int_{S_\C^{n-1}} \srel_k(\abs{\dotp{x, v_1}}^2, \ldots, \abs{\dotp{x, v_d}}^2) dx
    \ge \binom{d}{k} \frac{(n-1)!}{(k+n-1)!}
    \ge \binom{d}{k} (n+k-1)^{-k}.
  \end{align*}
  Where the last inequality comes from applying AM/GM. Since
  $\textsc{OptSoS}_k(\cA) = \bracket{\lambda^*/\binom{d}{k}}^{1/k}$, we get the
  desired bound.

  For the real case, we can bound the integration on the sphere with the
  following result (Theorem 2.2 of \cite{FrenkelPfaffianshafniansproducts2008}):
  For any $v_1, \ldots, v_k \in \R^n$ with $\norm{v_i} = 1$, the average of
  $\prod_{i=1}^k \dotp{v_i, x}^2$ on the unit sphere
  $\braces{x \in \R^n \mid \norm{x} = 1}$ is at least
    \begin{align*}
      \frac{\Gamma(n/2)}{2^k \Gamma(n/2 + k)} = \frac{1}{n(n+2)(n+4)\cdots(n+2k-2)}
      \ge (n+k-1)^{-k}.
    \end{align*}
  This combined with the rest of the argument in the complex case gets us the
  desired bound.
\end{proof}
Using Proposition \ref{prop:sos-relax-lower-bound} and the same upper bound on
the value of $\textsc{Opt}(\cA)$ in the proof of Theorem \ref{thm:int-gap-SDP},
we prove Theorem \ref{thm:int-gap-SoS}.
\subsection{Product of Nonnegative Forms}
We can also apply the same technique to produce low-degree relaxations for
product of nonnegative forms. Given a product of homogeneous polynomials
$p_1(x), \cdots, p_d(x)$ each of degree $2\ell$, we can apply Maclaurin's
inequality if the polynomials are non-negative. Hence we can obtain relaxations
of the form $\textsc{OptSoS}_k$ similar to the optimization problem in
Definition \ref{def:optsos}, replacing $\dotp{x, A_ix}$ with $p_i(x)$. This
problem involves solving a degree $k\ell$ Sum-of-Squares relaxation.

\section{Hardness}
\label{sec:hardness}
In this section we investigate the hardness of computing
$\textsc{Opt}(\cA)$. When $d$ is fixed, a result of Barvinok (Theorem 3.4 in
\cite{BarvinokFeasibilitytestingsystems1993}) provides a polynomial-time
algorithm for computing \eqref{eq:opt-prod-psd}. However we shall prove that
this problem is hard when $d = \Omega(n)$.
\begin{theorem} \label{thm:hardness}
  There exists a constant $\epsilon > 0$ so that for all $d = \Omega(n)$, it is NP-hard to
  approximate $\textsc{Opt}(\cA)$ defined in \eqref{eq:opt-prod-psd} better than a
  factor of $(1-\epsilon)^{1/d}$.
\end{theorem}
This is obtained by a reduction from \textsc{MaxCut}. In our proof we will use a
result by \cite{BermanTighterInapproximabilityResults1998}, showing that
\textsc{MaxCut} for 3-regular graphs is NP-hard to approximate better than a
factor of $\frac{331}{332}$ (for general graphs this factor can be improved to
$\frac{16}{17}$ \cite{HastadOptimalInapproximabilityResults2001}).

Let $G$ be a $3$-regular graph with unit edge weights and adjacency matrix $A$.
The matrix $Q_G = \frac{1}{2}(I - \frac{1}{3}A)\succeq 0$ is a scaling of the
graph Laplacian so that
$$ \textsc{MaxCut}(G) = \max_{x \in \{\pm 1/\sqrt{n} \}^n} x^\dagger Q_G x.$$
Next let $\lambda_{\max}(Q_G)$ be the largest eigenvalue of $Q_G$. A
result in spectral graph theory (see \cite{TrevisanMaxCutSmallest2012} for
example) shows that:
\begin{align} \label{eq:maxcut-eigen-approx}
  \frac{1}{2}\lambda_{\max}(Q_G) \le \textsc{MaxCut}(G) \le \lambda_{\max}(Q_G) \le 1.
\end{align}
Let $p_G(x) = x^\dagger Q_Gx \prod_{i=1}^n \paren{n x_i^2}^k$ be a product of
$d = nk+1$ PSD forms. The following optimization problem is equivalent to an
instance of \eqref{eq:opt-prod-psd}, after taking the $\frac{1}{d}$-th power:
\begin{align*}
  \textsc{Opt}(G) \defn \max_{\norm{x}_2 = 1} p_G(x).
\end{align*}
It is easy to show that $\textsc{Opt}(G)$ is a relaxation of
$\textsc{MaxCut}(G)$, as the feasible set $\norm{x}_2 = 1$ includes the boolean
cube $\{\pm 1/\sqrt{n}\}^n$, and $\prod_{i=1}^n \paren{n x_i^2}^k = 1$ on this
cube.
\begin{proposition}
  For any graph $G$, $\textsc{MaxCut}(G) \le \textsc{Opt}(G)$.
\end{proposition}
Next we claim that for all $\hat{x}$ on the sphere sufficiently far away from
the vertices of the boolean hypercube, the value of $p_G(x)$ is upper bounded by
$\textsc{MaxCut}(G)$, thus allowing us to restrict the feasible region to all
vectors $x$ that are close to a vertex of the hypercube.

\begin{figure}[ht]
  \centering
  \def\svgwidth{\columnwidth/3}
\begingroup%
  \makeatletter%
  \providecommand\color[2][]{%
    \errmessage{(Inkscape) Color is used for the text in Inkscape, but the package 'color.sty' is not loaded}%
    \renewcommand\color[2][]{}%
  }%
  \providecommand\transparent[1]{%
    \errmessage{(Inkscape) Transparency is used (non-zero) for the text in Inkscape, but the package 'transparent.sty' is not loaded}%
    \renewcommand\transparent[1]{}%
  }%
  \providecommand\rotatebox[2]{#2}%
  \newcommand*\fsize{\dimexpr\f@size pt\relax}%
  \newcommand*\lineheight[1]{\fontsize{\fsize}{#1\fsize}\selectfont}%
  \ifx\svgwidth\undefined%
    \setlength{\unitlength}{242.89460911bp}%
    \ifx\svgscale\undefined%
      \relax%
    \else%
      \setlength{\unitlength}{\unitlength * \real{\svgscale}}%
    \fi%
  \else%
    \setlength{\unitlength}{\svgwidth}%
  \fi%
  \global\let\svgwidth\undefined%
  \global\let\svgscale\undefined%
  \makeatother%
  \begin{picture}(1,1)%
    \lineheight{1}%
    \setlength\tabcolsep{0pt}%
    \put(0,0){\includegraphics[width=\unitlength,page=1]{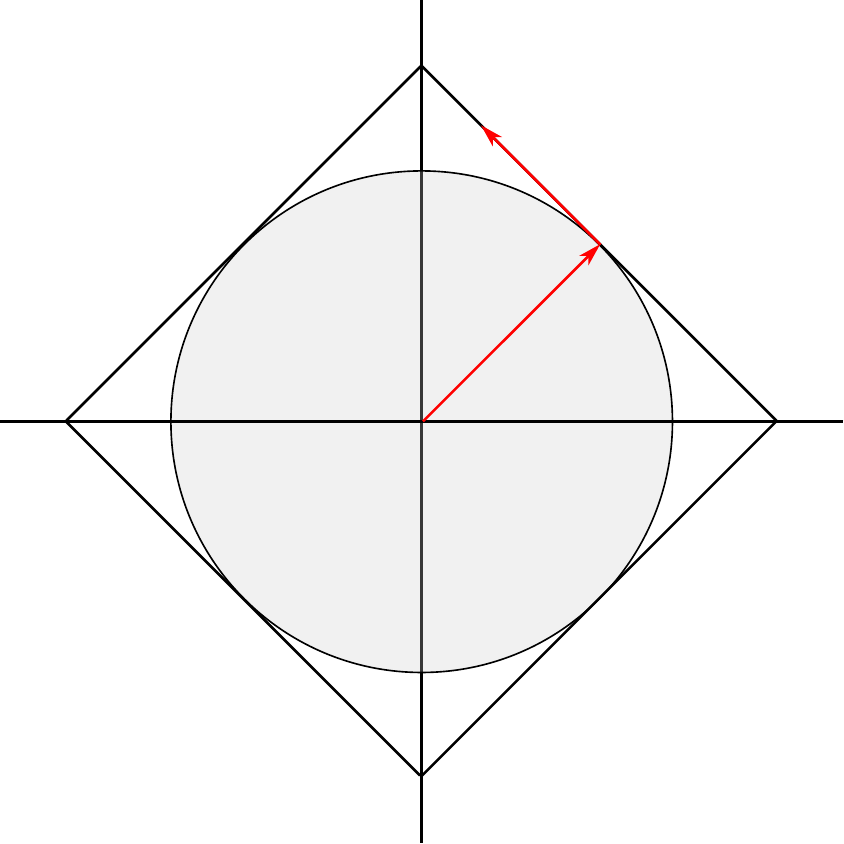}}%
    \put(0.60112184,0.8524902){\color[rgb]{1,0,0}\makebox(0,0)[lt]{\smash{\begin{tabular}[t]{l}$\Delta$\end{tabular}}}}%
    \put(0.72624745,0.72843411){\color[rgb]{1,0,0}\makebox(0,0)[lt]{\smash{\begin{tabular}[t]{l}$y = \frac{\mathbf{1}}{\sqrt{n}}$\end{tabular}}}}%
    \put(0,0){\includegraphics[width=\unitlength,page=2]{hardness_proof_src.pdf}}%
    \put(0.55347008,0.64339941){\color[rgb]{1,0,0}\makebox(0,0)[lt]{\smash{\begin{tabular}[t]{l}$x$\\\end{tabular}}}}%
  \end{picture}%
\endgroup%

  \caption{\label{fig:hardness} Illustration of the parameterization of the
    sphere we use in the proof of Proposition \ref{prop:hardness-restrict}.  }
\end{figure}
\begin{proposition} \label{prop:hardness-restrict} For any
  $\frac{2 \log 2}{k} \le \delta < n$, let $\eta =
  \frac{n}{\sqrt{n+\delta}}$. If $\norm{x}_2 = 1$ and $\norm{x}_1 \le \eta$,
  then $p_G(x) \le \textsc{MaxCut}(G) \le \textsc{Opt}(G)$. Letting
  $\mathcal{T}_{\delta} = \{ x\in \R^n \mid \norm{x}_2 = 1,\, \norm{x}_1 \ge
  \eta\}$, then $\textsc{Opt}(G) = \max_{x \in \mathcal{T}_\delta} p_G(x)$.
\end{proposition}
\begin{proof}
  We can write any $x$ on the sphere $\norm{x}_2 = 1$ as
  $x = (y + \Delta)/\norm{(y + \Delta)}$, where $y \in \{\pm 1/\sqrt{n}\}^n$ and
  $\Delta$ is orthogonal to $y$ (see Figure \ref{fig:hardness}). Let
  $y = \mathbf{1}/\sqrt{n}$ without loss of generality and
  $\norm{\Delta}_2^2 = \delta/n$.  Then any $x$ in the intersection of the
  sphere and non-negative orthant can be written as
  $$x = \sqrt{\frac{n}{n + \delta}}(\mathbf{1}/\sqrt{n} + \Delta)
  = \frac{\mathbf{1}/\sqrt{n} + \Delta}{\sqrt{1 + \delta/n}},$$ for some
  $\Delta$ where $\norm{\mathbf{1}/\sqrt{n} + \Delta}_1 = \sqrt{n}$ and
  $\delta \le n$. By construction, $\norm{x}_1 = \eta$. Next we bound the
  product
  \begin{align*}
    \prod_{i=1}^n \paren{n x_i^2}^k
    &= (1+\delta/n)^{-nk} n^{nk} \prod_{i=1}^n \abs{1/\sqrt{n} + \Delta_i}^{2k} \\
    &\le (1+\delta/n)^{-nk} n^{nk} \paren{\frac{1}{n}\sum_{i=1}^n \abs{1/\sqrt{n}+\Delta_i}}^{2nk} \\
    &= (1+\delta/n)^{-nk} \\
    &\le e^{-k\delta/2},
  \end{align*}
  where we have used the AM/GM inequality, the fact that
  $\norm{\mathbf{1}/\sqrt{n} + \Delta}_1 = \sqrt{n}$ and
  $(1+x/n)^{-n} \le e^{-x/2}$ for $0 \le x \le n$. Since
  $x^\dagger Q_G x \le \lambda_{\max}(Q_G) \le 2\textsc{MaxCut}(G)$, if
  $\delta \ge \frac{2 \log 2}{k}$, then for all $x$ in the nonnegative orthant
  where $\norm{x}_2 = 1$ and
  $\norm{x}_1 \ge \eta = \frac{n}{\sqrt{n + \delta}}$,
  $p_G(x) \le \textsc{MaxCut}(G)$. We can then repeat this argument for all
  other vertices of the hypercube. Geometrically $\mathcal{T}_\delta$ is defined
  as the union of spherical caps centered around the vertices of the hypercube
  $\{\pm 1/\sqrt{n} \}^n$. Thus for any $x \not \in \mathcal{T}_\delta$,
  $p_G(x) \le \textsc{MaxCut}(G)$ and we can restrict the optimization problem
  to $\mathcal{T}_\delta$.
\end{proof}

This restriction of the feasible set allows us to find an upper bound on
$\textsc{Opt}(G)$.
\begin{proposition}
  There exists a universal constant $C$ such that for all $k \ge C$,
  $\textsc{Opt}(G) < \frac{332}{331} \textsc{MaxCut}(G)$.
\end{proposition}
\begin{proof}
  Any $\hat x \in \mathcal{T}_\delta$ can be written as
  $\hat{x} = (y + \Delta)/\sqrt{1 + \delta/n}$, where $y \in \{\pm 1 /\sqrt{n}\}^n$,
  $\norm{\Delta}^2 \le \delta/n$ and $\dotp{\Delta, y} = 0$. Then
  \begin{align*}
    \hat{x}^\dagger Q_G \hat{x} \le (y + \Delta)^\dagger Q_G (y + \Delta)
    & \le \paren{\sqrt{\textsc{MaxCut}(G)} + \sqrt{\delta \lambda_{\max}(Q_G)/n}}^2 \\
    & \le \paren{\sqrt{\textsc{MaxCut}(G)} + \sqrt{\textsc{MaxCut}(G) 2\delta /n}}^2 \\
    & \le \textsc{MaxCut}(G) \paren{1 + \sqrt{2\delta/n}}^2
  \end{align*}
  where we used the bound in \eqref{eq:maxcut-eigen-approx}. We get the desired
  bound by choosing a large enough constant $k$ so that
  $\delta = \frac{2\log 2}{k}$ and $(1 + \sqrt{2\delta/n})^2 < \frac{332}{331}$
  for all $n$.
\end{proof}
This shows us that for a constant $k$, if we can find an algorithm that solves
$\textsc{Opt}(G)$, then we can also approximate $\textsc{MaxCut}(G)$ to within a
factor of $\frac{331}{332}$. However
\cite{BermanTighterInapproximabilityResults1998} showed that this is not
possible unless $P=NP$, thus completing the proof of Theorem \ref{thm:hardness}.

\section{Conclusion}
In this paper we studied the problem of maximizing the product of non-negative
forms over the sphere. Even though the objective is a high degree dense
polynomial on the sphere, we leveraged its compact representation as a product
of low degree polynomials formulate a series of computationally efficient
relaxations. We then provided bounds on the quality of these relaxations and
showed that they are much better than known bounds for approximating general
polynomial optimization.

A few intriguing questions remain. Although we showed a partial order for the
values of relaxations in Section \ref{sec:higher-degree-derivation}, it remains
to prove that the values of $\textsc{OptSOS}_k(\cA)$ are monotone for increasing
values of $k$. Numerical experiments suggest that this is the case. Another open
problem is to extend the analysis of the performance ratio of the Sum-of-Squares
relaxation in section \ref{sec:rounding-alg} to find a bound on its
approximation ratio. Answering these questions may require proving identities
involving products of pseudoexpectations.

The main tools in formulating the low degree relaxations in this paper are
algebraic identities such as the AM/GM and Maclaurin's inequalities, that bounds
the objective and at the same time reduces the polynomial's degree. This idea
may also be applied to other optimization problems with compact representation.

\bibliographystyle{amsalpha}
\bibliography{../MyLibraryTEX}

\appendix

\section{Expected Log of Generalized Chi-squared Distribution}
\label{sec:expected-log-chi-squared}
Given $\lambda_1, \ldots, \lambda_n > 0$ and let $z_i \sim \cnormal(0, 1)$ be
i.i.d. complex Gaussians, we wish to find:
\begin{align*}
  \Ex\bracket{\log \paren{\sum_i \lambda_i \abs{z_i}^2}}
\end{align*}
Using equation (11) from \cite{GaoDeterminantRepresentationDistribution2000}, we
know that the density of the random variable
$Z = \sum_i \lambda_i \abs{z_i}^2$ is:
\begin{align*}
  f(z) = (-1)^{n-1} \sum_{i=1}^n \frac{\lambda_i^{n-2} \exp(-z/\lambda_i)}
                                      {\prod_{j \ne i} (\lambda_j - \lambda_i)}
\end{align*}
Suppose $\lambda_i$ are distinct, using the integral
$\int_0^\infty \log(z) \exp(-z/\lambda_i) = \lambda_i(-\gamma + \log
\lambda_i)$, we get:
\begin{align*}
  \Ex[\log(Z)] &= (-1)^{n-1} \sum_{i=1}^n \frac{\lambda_i^{n-1} (-\gamma + \log \lambda_i)}
                 {\prod_{j \ne i} (\lambda_j - \lambda_i)} \\
               &= -\gamma + (-1)^{n-1} \sum_{i=1}^n \frac{\lambda_i^{n-1} \log \lambda_i}
                 {\prod_{j \ne i} (\lambda_j - \lambda_i)}.
\end{align*}
The identity in the last step can be proved using different representations of
the determinant of a Vandermonde matrix. The sum can be represented as a ratio
of determinants. Let
\begin{align*}
  V = \bmat{1 & \lambda_1 & \cdots & \lambda_1^{n-1} \\
            1 & \lambda_2 & \cdots & \lambda_2^{n-1} \\
            \vdots & \vdots & \ddots & \vdots \\
            1 & \lambda_n & \cdots & \lambda_n^{n-1}} \quad \text{and} \quad
  \bar{V} = \bmat{1 & \lambda_1 & \cdots & \lambda_1^{n-1} \log \lambda_1 \\
                  1 & \lambda_2 & \cdots & \lambda_2^{n-1} \log \lambda_2 \\
                  \vdots & \vdots & \ddots & \vdots \\
                  1 & \lambda_n & \cdots & \lambda_n^{n-1} \log \lambda_n}.
\end{align*}
Then
\begin{align*}
  \Ex[\log(Z)] = -\gamma + \frac{\det(\bar{V})}{\det(V)}.
\end{align*}

Now suppose some of the $\lambda_i$ are repeated, then we can determine the pdf
of $Z$ using results from Section II of
\cite{BjornsonExploitingQuantizedChannel2009}. In particular, if
$\lambda_1 = \lambda$ and $\lambda_2, \ldots, \lambda_n = \eps$, then
\begin{align*}
  f(z) = \frac{1}{\lambda\eps^{n-1}}\paren{
  \frac{e^{-z/\lambda}}{\paren{1/\eps-1/\lambda}^{n-1}} +
  \sum_{\ell=1}^{n-1} \frac{(-1)^{\ell+1}x^{n-1-\ell} e^{-x/\eps}}
  {(n-1-\ell)!(1/\lambda - 1/\eps)^\ell}}.
\end{align*}
Using the integral (where $b\ge 1$ and $a > 0$)
\begin{align*}
  \int_0^\infty x^{b-1} e^{-x/a} \log x \, dx  = a^b(b-1)!(\log(a) + \psi(b)),
\end{align*}
we can derive a closed form expression for $\Ex[\log(Z)]$:
\begin{align*}
  \Ex[\log(Z)] &= \frac{1}{\lambda\eps^{n-1}}\paren{
  \frac{-\gamma + \log \lambda}{\paren{1/\eps-1/\lambda}^{n-1}} +
  \sum_{\ell=1}^{n-1} \frac{(-1)^{\ell+1} \eps^{n-\ell} (\log \eps + \psi(n-\ell))}
                           {(1/\lambda - 1/\eps)^\ell}} \\
  &= \frac{\lambda^{n-1}(-\gamma + \log \lambda)}{(\lambda - \eps)^{n-1}} -
    \sum_{\ell=1}^{n-1} \frac{\eps \lambda^{\ell-1}(\log \eps + \psi(n-\ell))}{(\lambda - \eps)^\ell}.
\end{align*}

\section{Proof of Proposition \ref{prop:max-monomial-sphere}}
\label{sec:prop-monomial-proof}
From \cite{FollandHowIntegratePolynomial2001} we know that given the monomial
$x^\beta = \prod_{i=1}^n x_i^{\beta_i}$, its integral over the real sphere
$\mathcal{S}^{n-1}$ can be computed as follows:
\begin{align*}
  \int_{\mathcal{S}^{n-1}} x^{\beta} \,dx =
  \frac{2 \Gamma(\gamma_1) \cdots \Gamma(\gamma_n)}{\Gamma(\gamma_1 + \cdots + \gamma_n)},
\end{align*}
where $\gamma_i = \frac{1}{2}(\beta_i + 1)$. Next let $d = \sum_i \beta_i$
and $k \ge 1$ be an integer. We use Stirling's approximation and take the
limit
\begin{align*}
  \max_{\norm{x} = 1} x^{2\beta} = \lim_{k \rightarrow \infty}
  \paren{\int_{\mathcal{S}^{n-1}} x^{2k\beta} dx}^{1/k}
  &= \lim_{k \rightarrow \infty}
    \pfrac{2 \prod_{i=1}^n \Gamma(k\beta_1 + 1/2)}{\Gamma(kd + n/2)}^{1/k} \\
  &= \lim_{k \rightarrow \infty} \frac{\prod_{i=1}^n (k\beta_i -1/2)^{\beta_i}}{(kd + n/2-1)^d} \\
  &= \frac{\prod_{i=1}^n {\beta_i}^{\beta_i}}{d^d}.
\end{align*}

\end{document}